\documentclass[a4paper,11pt,leqno,english]{smfart}
\usepackage{aeguill}
\usepackage{enumerate}
\usepackage{amssymb,amsmath,latexsym,amsthm, smfthm}
\usepackage[T1]{fontenc}
\usepackage{geometry}   
\usepackage{url}   
\usepackage[frenchb, english]{babel}
\usepackage[utf8]{inputenc}   
\usepackage{mathrsfs}
\usepackage{xcolor}
\usepackage{comment}
\definecolor{violet}{rgb}{0.0,0.2,0.7}
\definecolor{rouge2}{rgb}{0.8,0.0,0.2}
\usepackage{hyperref}
\usepackage{smfhyperref}
\hypersetup{
    bookmarks=true,         
    unicode=false,          
    pdftoolbar=true,        
    pdfmenubar=true,        
    pdffitwindow=false,     
    pdfstartview={FitH},    
    pdftitle={},    
    pdfauthor={},     
    colorlinks=true,       
   linkcolor=violet,          
    citecolor=violet,        
    filecolor=black,      
    urlcolor=cyan}           
\usepackage{geometry}

\newcommand{\R}{\mathbb{R}}

\newcommand{\DD}{\mathbb D}

\renewcommand{\d}{\partial}

\newcommand{\vp}{\varphi}

\renewcommand{\O}{\mathcal{O}}
\newcommand{\Ox}{\mathcal{O}_{X}}
\newcommand{\ep}{\varepsilon}
\renewcommand{\epsilon}{\varepsilon}

\newcommand{\la}{\langle}

\newcommand{\ra}{\rangle}

\renewcommand{\ge}{\geqslant}
\renewcommand{\le}{\leqslant}
\renewcommand{\leq}{\leqslant}

\newcommand{\Ric}{\mathrm{Ric} \,}
\newcommand{\om}{\omega}

\newcommand{\ddc}{dd^c}

\newcommand{\vpe}{\varphi_{\varepsilon}}

\newcommand{\ome}{\om_{\varepsilon}}
\newcommand{\omvp}{\omega_{\varphi}}
\newcommand{\omt}{\om_{t}}

\newcommand{\omvpe}{\om_{\vp_{\ep}}}
\newcommand{\re}{\rho_{\ep}}
\newcommand{\Supp}{\mathrm {Supp}}
\newcommand{\tr}{\mathrm{tr}}

\newcommand{\vol}{\mathrm{vol}}

\newcommand{\Tht}{\Theta_{h^{\om}_{X_t}}(K_{X_t})}

\renewcommand{\tt}{\theta_{i,t}}

\newcommand{\ssm}{\smallsetminus}

\newcommand{\ombte}{\om_{\beta,t, \ep}}
\newcommand{\omte}{\om_{t, \ep}}
\newcommand{\ombe}{\om_{\beta, \ep}}
\newcommand{\vep}{v_{\ep}}
\newcommand{\bvep}{{\bar{v}}_{\ep}}

\newcommand{\vph}{v_{\vp}}
\newcommand{\Le}{L_{\ep}}

\newtheorem*{thms}{Theorem}

\numberwithin{equation}{section}
\begin{document}
\title{Families of conic Kähler-Einstein metrics}

\author{Henri Guenancia}
\address{Department of Mathematics \\
Stony Brook University, Stony Brook, NY 11794-3651}
\email{guenancia@math.sunysb.edu}
\urladdr{www.math.sunysb.edu/{~}guenancia}

 \thanks{The author is partially supported by NSF Grant DMS-1510214}

\date{\today}

\begin{abstract}
Let $p:X\to Y$ be an holomorphic surjective map between compact Kähler manifolds and let $D$ be an effective divisor on $X$ with generically simple normal crossings support and coefficients in $(0,1)$. Provided that the adjoint canonical bundle $K_{X_y}+D_y$ of the generic fiber is ample, we show that the current obtained by glueing the fiberwise conic Kähler-Einstein metrics on the regular locus of the fibration is positive. Moreover, we prove that this current is bounded outside the divisor and that it extends to a positive current on $X$. 
\end{abstract}
\maketitle
\tableofcontents

\section*{Introduction}

Let $p:X \to Y$ be a holomorphic surjective map between two compact Kähler manifolds $X$ and $Y$, and let $D=\sum_{k=1}^r D_k$ be a reduced divisor on $X$ with generically simple normal crossings and mapping surjectively to $Y$ by $p$. We denote by $W \subset Y$ the minimal analytic subset of $Y$ such that if $X_0:=p^{-1}(Y\ssm W)$, then every fiber $X_y$ of $p_{|X^0}$ is smooth, $D_{|X_y}$ has simple normal crossings (and therefore is transverse to $X_y$).  Finally, let $\{\gamma\}\in H^{1,1}(X,\R)$ be a real cohomology class containing a smooth semipositive form $\gamma$. 

Now we assume that for a generic $y\in Y$ and a set of numbers $\beta_1, \ldots, \beta_r \in (0,1)$, the cohomology class 
$$c_1(K_{X_y}+\sum_{k=1}^r (1-\beta_k){D_k}_{|{X_y}})+ \{\gamma\}_{|X_y}$$
is Kähler. By \cite{GP}, there exists on each such fiber $X_y$ a unique (twisted) conic Kähler-Einstein metric $\om_y$ with cone angles $2\pi \beta_k$ along ${D_k}_{|X_y}$ satisfying:
\begin{equation}
\label{KE}
\Ric \om_y = -\om_y+\gamma_{|X_y}+ \sum_{k=1}^r(1-\beta_k) [{D_k}_{|X_y}]
\end{equation}
Over $X_0$, it is possible to glue the fiberwise conic Kähler-Einstein metrics $\om_y$ to get a current $\rho\in c_1(K_{X/Y}+ \sum_{k=1}^r(1-\beta_r) D_k)+\{\gamma\}$ with locally bounded potentials. A priori, $\rho$ need not have regularity or positivity in the directions transverse to the fibers, and it need not extend to $X$. Our main theorem addresses these questions:

\vspace{4mm}
\begin{thms}
Let $p:X\to Y$ a holomorphic surjective map between compact Kähler manifolds, $D=\sum_{k=1}^r D_k$ a reduced divisor with generically simple normal crossings, $\{\gamma\}\in H^{1,1}(X,\R)$ a semipositive class and $\beta_1, \ldots, \beta_r \in (0,1)$ such that the cohomology class $c_1(K_{X_y}+\sum_{k=1}^r (1-\beta_k){D_k}_{|{X_y}})+ \{\gamma\}_{|X_y}$ is Kähler for every  $y\in Y\ssm W$. Then the fiberwise twisted conic Kähler-Einstein metrics $\om_y$ satisfying \eqref{KE} can be glued to define a current $\rho$ on $X_0$ such that:
\begin{enumerate}
\item[$\bullet$] $\rho$ is positive
\item[$\bullet$] $\rho$ is bounded outside $D$
\item[$\bullet$] $\rho$ extends to $X$ as a closed positive current in $ c_1(K_{X/Y}+ \sum_{k=1}^r(1-\beta_r) D_k)+\{\gamma\}$.
\end{enumerate}
In particular, the cohomology class $ c_1(K_{X/Y}+ \sum_{k=1}^r(1-\beta_k) D_k)+\{\gamma\}$  is pseudoeffective. 
\end{thms}
\vspace{4mm}


The first two items can be summarized by saying that if $\om$ is a Kähler form on $X$, then for any relatively compact subset $\Omega \Subset X_0\ssm D$, there exists a constant $C= C(\Omega)$ such that
$$0\le \rho \le C \om$$
holds on $\Omega$. In particular, the local coefficients of $\rho$ are locally bounded functions outside $D$. Also, the local potentials of $\rho$ on that set are $\mathscr C^{1,\alpha}$-regular for any $\alpha<1$. We refer to Corollary \ref{cor:reg} for a slightly refined statement. \\
 
This Theorem is the "conic" analogue of the main theorem of \cite{Paun12} where $D=0$. We will follow the same strategy but the analysis becomes significantly more subtle as we need to deal with conic metrics, whose regularity properties are far too weak to simply follow the lines of \cite{Paun12}, cf next paragraph. From an algebraic point of view, this is reflected by the difference between dealing with semiample line bundles and merely effective ones. The last item (extension of $\rho$ to $X$) is however very similar.  \\

If $\{\gamma\}$ is the cohomology class of a line bundle (up to a scalar), then the pseudoeffectivity of $ c_1(K_{X/Y}+ \sum_{k=1}^r(1-\beta_r) D_k)+ \{\gamma\}$ is a consequence of the far more general theorem \cite{BP} about the psh variation of Bergman kernels. However, it is very interesting to understand the variation of Kähler-Einstein metrics (rather than Bergman kernels), as these canonical objects should detect the variation in moduli of the family.  

\vspace{5mm}

%
%

\begin{center}
\textbf{Outline of the proof.}
\end{center}

\vspace{5mm}

Let us first recall the strategy of the proof in the case where $D=0$ (\cite[Theorem 1.1]{Paun12}). First, one proves that the current $\rho$ is positive on $X_0$, and then one proves that it extends to a positive current on $X$ using a refined version of Ohsawa Takegoshi theorem \cite{BP2}. 

To show that $\rho$ is positive on $X_0$, it is crucial to first apply the implicit function theorem to get that the fiberwise (twisted) Kähler-Einstein varies smoothly on $X_0$. After observing \cite[Remark 3.1]{Paun12} that it is enough to consider the case where $Y=\mathbb D$, everything comes down to showing that the smooth function $c(\rho)$ is non-negative (the argument will actually show that $c(\rho)>0$). To do so, one works over a fixed fiber $X_t$ and derives an elliptic equation satisfied by $c(\rho)$: 
$$(-\Delta_{\rho}+\mathrm{Id})c(\rho)= |\bar \d v_{\rho}|^2+\gamma(v_{\rho}, v_{\rho})$$
As the right hand side is positive and the operator $-\Delta_{\rho}+\mathrm{Id}$ is positive, one get the expected result, cf \cite[\S 3]{Schum}. One can notice that an application of the maximum principle would work just as well for that purpose. \\

The general case with a boundary divisor $D$ presents some major new difficulties. The general strategy will be the same: first show that $\rho$ is a positive current on $X_0$, and then extend it to $X$; however, it will require an additional effort to show that $\rho$ is smooth on $X_0\ssm D$. This last step surprizingly relies on the positivity of $c(\rho)$. The second part is essentially similar to the one for $D=0$, cf \S \ref{extension}. The first part, however, becomes significantly more involved. The main reason is that the lack of (known) regularity for conic Kähler-Einstein metrics prevents us from applying the implicit function theorem directly to the fiberwise conic KE metrics to get their smooth variation (in the conic sense). Without this, one cannot make sense of $c(\rho)$, and the whole argument collapses. 

To circumvent this difficulty, we will work instead with approximate conic Kähler-Einstein metrics $\ome$ on each fiber; these are smooth metrics converging fiberwise to the conic Kähler-Einstein metrics. Glueing them yields a smooth $(1,1)$-form $\rho_{\ep}$ on $X_0$, solution on each fiber of an elliptic equation of the form:
  $$(-\Delta_{\re}+\mathrm{Id})c(\re)= |\bar \d v_{\re}|^2+\gamma(v_{\re}, v_{\re})+\Theta_{\ep}(v_{\re}, v_{\re})$$
where $\Theta_{\ep}$ is a smooth approximation of the current $\sum(1-\beta_k) [D_k]$. The first two terms in the right hand side are non-negative but unfortunately $\Theta_{\ep}$ is \textit{not} semipositive in general (cf e.g. \cite[Remark 1.1]{C2C}), so that we do not get the semipositivity of $c(\re)$. Rather, we have a lower bound of $\inf_{X_t} c(\re)$ that essentially looks like $-\int_{X_t}\frac{\ep^2}{|s|^2+\ep^2}\cdotp |v_{\re}|^2_{\om}\rho_{\ep}^n$, where we assumed to simplify that $D=(s=0)$ was smooth; here $\om$ is a smooth metric on $X_t$. Actually the integral is more complicated and involves an upper bound of the heat kernel, cf \eqref{infc}. 
 Most of the effort is then concentrated on showing that this last integral tends to $0$ when $\ep \to 0$. 

Looking at the expression of $v_{\re}$ recalled in the previous paragraph, one sees that the estimate one needs on the (local) potential $\phi_{\ep}$ is a uniform $L^2$ estimate for $\nabla \d_t \phi_{\ep}$ (actually one will need slightly more, cf Proposition \ref{support}). Because of the global nature of the problem, it seems to us that the methods to get local estimates (on the fibers) are bound to fail, and this is the source of a lot of complications: as $\frac{\d}{\d t} $ is only defined locally, one will have to work with one of its lifts $\vep$ with respect to some approximate conic metric (so $\vep$ is different from $ v_{\re}$ which we don't know well enough to carry out any precise computations); $\vep$ will necessarily carry some singularities along the normal directions to the divisors, and this involves a great deal of complication in our analysis. The general scheme of the proof is the following: \\
\begin{enumerate}
\item[$Step \, 1.$] Get various estimates on $\phi_{\ep}$ and its derivatives in the fiber directions (these estimates need to be uniform for nearby fibers): $L^{\infty}$, gradient estimates (this is done by generalizing \cite{Bl}), Laplacian (this is mostly \cite{GP}) as well as Sobolev and Poincaré constants estimates. Also, one recalls an upper bound on the heat kernel that is crucial to control the negativity of $c(\re)$.\\
\item[$Step \, 2.$] Get a $L^2$ estimate on $\vep \cdotp \phi_{\ep}$: this is done by differentiating the Monge-Ampère equation satisfied by $\phi_{\ep}$. Because $\vep$ is not holomorphic in the fiber directions, the linear equation one gets involves a lot of auxiliary terms. Deduce an $L^{\infty}$ estimate on $\vep \cdotp \phi_{\ep}$ and a priori estimates at any given order on $\d_t \phi_{\ep}$ outside $D$.\\
\item[$Step \, 3.$] From the previous step, one gets an $L^2$ estimate for $\nabla(\vep \cdotp \phi_{\ep})$. One can use it to eventually get an $L^2$ estimate for $\nabla \d_t \phi_{\ep}$, and from there a $L^2$ estimate for $v_{\re}$. This enables to get a lower bound on $c(\re)$. This lower bound good enough to show that $\rho$ is positive on $X\ssm D$, from which the global positivity follows easily. \\
\item[$Step \, 4.$] Plugging that new input in the equation satisfied by $c(\rho_{\ep})$, one gets a uniform upper bound for $c(\rho_{\ep})$; one can combine this with the results from Step 2 to get estimates at any given order on $\d^2_{t \bar t}\phi_{\ep}$ outside $D$, hence the boundedness of $\rho$.\\ 
\item[$Step \, 5.$] Extend $\rho$ to $X$ by following the argument in \cite[\S 3.3]{Paun12}. \\
\end{enumerate}

\vspace{3mm}

\noindent
\textbf{Acknowledgements.} I would like to thank Mihai P\u{a}un for the interest he showed in this work and the numerous discussions we had about it. I am grateful to Vincent Guedj who pointed out a mistake in an earlier version of the text.

\section{Preliminaries}
\subsection{Differential geometric aspects of families of manifolds}
\label{diff} 
We give here a recollection of some important objects useful for the differential-geometric study of families of Kähler metrics. We refer to e.g. \cite{BP, Bern11, Paun12} for more details. 
\vspace{2mm}

\noindent
\textit{Singular metric on} $K_{X/Y}$. 
Given a holomorphic surjective map $p:X\to Y$ of relative dimension $n$ and a smooth closed $(1,1)$-form $\rho$ on $X$ such that $\rho$ restricts to a Kähler form to each regular fiber $X_y$, $y\in Y\ssm W$, on can cook up a singular hermitian metric $h^{\rho}_{X/Y}$ on the relative canonical bundle $K_{X/Y}:=K_X-p^*K_Y$ in the following way. 

Consider $x\in X$ and a neighborhood $U\ni x$ with trivializing coordinates $(z_1, \ldots, z_{n+m})$ near $x$. Write $y=p(x)$, and choose a system of coordinates $(t_1, \ldots, t_m)$ near $y$. The one can define the local weight $\phi_U$ of $h_{X/Y}^{\rho}$ by the formula:
$$\rho^n \wedge \bigwedge_{k=1}^m p^*(i dt_k \wedge d \bar t_{k}) =e^{\phi_{U}} \bigwedge_{k=1}^{n+m} i dz_k \wedge d \bar z_{k}$$
Of course $\phi_U$ (hence $h_{X/Y}^{\rho}$) is smooth over $X_0$, but near the singular fibers the left hand side will vanish, which translates into $\phi_U$ being $-\infty$ (but it is still defined as a singular hermitian metric). We refer to \cite[\S 3.1]{Paun12} for more details. We write $\Theta_{h^{\rho}_{X/Y}}(K_{X/Y})$ for the Chern curvature of the (singular) hermitian line bundle $(K_{X/Y}, h^{\rho}_{X/Y})$.  \\

\noindent
\textit{The canonical lift of} $\frac{\d}{\d t}$. 
Assume for now on that $p:X\to \mathbb D$ is a smooth fibration onto the unit disk of $\mathbb C$, and let us recall the construction of the so-called canonical lift of $\frac{\d}{\d t}$ to $X$. Given a smooth closed $(1,1)$-form $\rho$ on $X$ such that $\rho$ is a Kähler form in restriction to the fibers, there is a way to lift $\frac{\d}{\d t}$ canonically with respect to $\rho$ (cf \cite{Siu, Schum, Bern11}). This means that one can construct a unique vector field $v_{\rho}$ on $X$ such that $p_*{v_{\rho}}=\frac{\d}{\d t}$ and $\la v_{\rho}, w\ra_{\rho}=0$ for any vector $w\in T^{1,0}_{X_t}$. 

\noindent
Choosing local coordinates $(z,t)$ on $X$ such that $p(z,t)=t$, and a local potential $\varphi$ of $\rho$ on this chart, one first introduces the vector field $w_{\rho}$ by the relation 
$$\iota_{w_{\rho}}(\ddc \vp) = \bar \d \dot{\vp}_t$$
where $\dot{\vp}_t:=\frac{\d \vp}{\d t}$; in other words,  $w_{\rho}$ is the complex gradient of $\dot{\vp}_t$ with respect to $\rho_{|X_t}$. Then one can prove that 
$$v_{\rho}:=\frac{\d}{\d t} - w_{\rho}$$
is a well-defined smooth vector field over $X$ that lifts $\frac{\d}{\d t}$ in the canonical way explained above. If $\rho$ is locally given by 
$$\rho_{t\bar t}\,idt \wedge d\bar t+ \sum_{\alpha} \rho_{\alpha \bar t }\, idz_{\alpha} \wedge d\bar t+ \sum_{\alpha} {\rho}_{t \bar \alpha} \, idt \wedge d \bar z_{\alpha}+\sum_{\alpha, \beta} \rho_{\alpha, \bar \beta} \,idz_{\alpha} \wedge d\bar z_{\beta} $$
then one has $$v_{\rho}=\frac{\d}{\d t} -\sum_{\alpha, \beta} \rho^{\bar \beta \alpha }\rho_{t \bar \beta}\frac{\d}{\d z_{\alpha}}$$ 

\noindent
\textit{Geodesic curvature}. Given the situation above of a smooth fibration $p:X\to \mathbb D$, one defines the geodesic curvature of $\rho$ to be the smooth function $c(\rho)$ defined on $X$ by $c(\rho):=\la v_{\rho}, v_{\rho}\ra_{\rho}$. Equivalently, $c(\rho)$ can be defined by the relation
$$\rho^{n+1}= c(\rho) \rho^n \wedge idt \wedge d\bar t$$
It measure the potential lack of positivity of $\rho$ on $X$ (essentially, it is the eigenvalue of $\rho$ in the transverse directions). Expressed in terms of a local potential $\varphi$ of $\rho$, one finds $c(\rho)= \d^2_{t\bar t}\vp- |\bar \d_z\d_t \vp|	^2$. Also, one has the formula \cite[Lemma 4.1]{Bern11}:
$$\iota_{v_{\rho}}\rho= -c(\rho)\, i d\bar t$$
and in particular $v_{\rho}$ is a Killing vector field (for $\rho_{|X_t}$) if $c(\rho)$ vanishes on $X_t$.
Finally, one can find an expression in coordinates: 
$$c(\rho)= \rho_{t\bar t}- \sum_{\alpha, \beta} \rho^{\bar \alpha \beta}\rho_{t\bar \alpha}\rho_{\beta \bar t}$$

\subsection{Metrics with conic singularities}
\label{conic}
Let $X$ be a complex Kähler manifold. 

A divisor $D$ (formal $\R$-linear combinations of hypersurfaces) is said to have simple normal crossing support if near any point in its support, $\Supp(D)$ is given by $(z_1 \,\cdots \, z_d=0)$ for some holomorphic system of coordinates $(z_i)$, and if all its irreducible components (for the Zariski topology) are smooth.  

Given a $\R$-divisor $D= \sum (1-\beta_k)D_k$ with simple normal crossing support such that $\beta_k\in (0,1)$ for all $k$, we can associate the notion of Kähler metric with \textit{conic singularities} along $D$: it is a Kähler metric $\om$ on $X\setminus(\cup D_k)$ which is quasi-isometric to the model metric with conic singularities: more precisely, near each point $p\in \Supp(D)$ where $(X,D)$ is isomorphic to the pair $(\mathbb D^n, \sum_{k=1}^r (1-\beta_k)[z_k=0])$ up to relabelling the $\beta_k$'s, we ask $\om$ to satisfy under this identification:
\[C^{-1} \om_{\rm cone} \le \om \le C \om_{\rm cone}\]
for some constant $C>0$, and where 
\[\om_{\rm cone}:=\sum_{k=1}^d \frac{1}{|z_k|^{2(1-\beta_k)}}i dz_k\wedge d\bar z_k +\sum_{k=d+1}^n i dz_k\wedge d\bar z_k \] 
is the model cone metric with cone angles $2\pi \beta_k$ along $(z_k=0)$. \\

This type of metrics arise naturally in the theory of Kähler-Einstein metrics for pairs. More precisely, one has the following theorem \cite{GP}

\begin{thms}\cite{GP}
Let $X$ be a compact Kähler manifold, and $D= \sum (1-\beta_k)[s_k=0]$ a divisor with simple normal crossing support. Let $\om$ be a Kähler metric on $X$, $dV$ a smooth volume form, and let $\mu \in \R$. Then any weak solution $\omvp=\om + \ddc \vp$ with $\vp \in L^{\infty}(X)$ of 
\[(\om+\ddc \vp)^n = \frac{e^{\mu \vp}dV}{\prod |s_k|^{2(1-\beta_k)}} \] 
has conic singularities along $D$. \\
\end{thms}

Moreover, it is not difficult to see that if $\mu>0$, then such a weak solution always exists (cf \cite{CGP}). Given this statement, constructing conic Kähler-Einstein metrics boils down to checking some cohomological condition. More precisely, if $\gamma$ is a smooth representative of a class $\{\gamma\}$ such that $c_1(K_X+D)+\{\gamma\}$ is a Kähler class, then there exists a smooth volume form $dV$ such that $\om:=-\Ric(dV)+\sum (1-\beta_k) \Theta_{h_k}(D_k)+ \gamma$ is a Kähler form, where $D_k=[s_k=0]$ and $\Theta_{h_k}(D_k)$ is the Chern curvature of a smooth hermitian metric $h_k$ (denoted abusively $|\cdotp |$ in the theorem above) on $\mathcal O_X(D_k)$. Solving the Monge-Ampère equation above with this data and $\mu=1$ produces then a metric $\omvp$ with conic singularities along $D$, and such that in the sense of currents: 
$$\Ric(\omvp)= -\omvp+\sum (1-\beta_k)[D_k]+\gamma$$
We call such a metric a twisted conic Kähler-Einstein metric.

\subsection{Setting}

Because positivity (as well as smoothness) can be checked locally on one-dimensional bases, we will assume until the end (except for \S \ref{extension}) that $X$ is a Kähler manifold equipped with a smooth and proper morphism $p:X\to \mathbb D$ to the unit disk in $\mathbb C$. Moreover, the assumptions allows us to assume that $D=\sum_{i=1}^N D_i\subset X$ has simple normal crossings and  that each of its components is transverse to every fiber of $p$. Write $X_t:=p^{-1}(t)$, $D_{i,t}:={D_i}_{|X_{t}}$ and suppose that for any $t\in \mathbb D$ we have $c_1(K_{X_t}+\sum_i(1-\beta_i) D_{i,t})+\{ \gamma \}_{|X_{t}}$ is a Kähler class for some numbers $\beta_i \in (0,1)$. Because of the assumptions, every fiber $X_t$ is smooth, so the potential conflict of notation between the fiber $p^{-1}(0)$ and the regular locus of the fibration will hopefully not cause any confusion.

One chooses a background Kähler metric $\omega$ on $X$; it induces a smooth $(1,1)$ form $\Theta_{h^{\om}_{X/\DD}}(K_{X/\DD})$ over $X$ belonging to $c_1(K_{X/\DD})$, cf \S \ref{diff}. Fix smooth hermitian metrics on $\Ox(D_i)$, their Chern curvature forms $\theta_{D_i}$ are smooth $(1,1)$forms representing $c_1(D_i)$. We write $\tt:={\theta_{D_i}}_{|X_t}$, and $\gamma_t:=\gamma_{|X_t}$. One chooses sections $s_i$ cutting out $D_i$. Thanks to \cite{GP} one can solve on each fiber $X_t$ the equation:
$$ (\Tht+\sum(1-\beta_i)\tt+\gamma_t+\ddc \vp_t)^n = \frac{e^{\vp_t}\om^n}{\prod_i|s_i|^{2(1-\beta_i)}}$$
to obtain a conic metric $\omt$ on $X_t$ such that
\begin{equation}
\label{eq:0}
\Ric \omt= -\omt + \sum_i(1-\beta_i)[D_{i,t}] +\gamma_t
\end{equation}
We can glue the potentials $\vp_t$ to get a function $\vp$ on $X$. It is not difficult to see that $\vp$ is locally bounded (cf first item of Proposition \ref{prop:est}), hence one gets a closed $(1,1)$ current $$\rho:=\Theta_{h^{\om}_{X/\DD}}(K_{X/\DD})+\sum_i(1-\beta_i) \theta_{D_i}+\gamma + \ddc \vp$$ on $X$ belonging to $c_1(K_{X/\DD}+\sum_i (1-\beta_i)D_i)+\{\gamma\}$. This is the current we are interested in; more precisely, we want to show that $\rho$ is positive, smooth outside $D$ and dominated by a conic metric on $X$. 

\section{Estimates in the fiber directions}
\subsection{Approximate fiberwise conic metrics}
As recalled in the introductory sections, working directly with the fiberwise Kähler-Einstein conic metrics (hence with $\rho$) involves a great deal of complications. Instead, one will proceed by approximation and try to get uniform estimates along the process. 

Let $\ep>0$; for each $t\in \DD$, one can solve the equation
\begin{equation}
\label{eq:eps0}
 (\Tht+\sum_i(1-\beta_i)\tt+\gamma_t+ \ddc \vp_{t,\ep})^n = \frac{e^{\vp_{t, \ep}}\om^n}{\prod_i(|s_i|^{2}+\ep^2)^{1-\beta_i}}
 \end{equation}
 thanks to Aubin-Yau theorem \cite{Aubin, Yau78}. It yields a smooth metric $\omte$ on $X_t$ such that 
\begin{equation}
\label{eq:eps}
\Ric \omte = -\omte +\sum_i(1-\beta_i) w_{i,\ep}+\gamma_t
\end{equation}
where $w_{i,\ep}=\tt+\ddc \log(|s_i|^2+\ep^2) = \frac{\ep^2 |D's_i|^2}{(|s_i|^2+\ep^2)^2}+\frac{\ep^2}{|s_i|^2+\ep^2} \, \tt$ approximates the current of integration on $D_{i,t}$. Here again, one can glue the potentials $\vp_{t,\ep}$ to get a function $\vpe$ on $X$, which in turns defines a current $\rho_{\ep}:=\Tht+\sum_i(1-\beta_i)\tt+\gamma+\ddc \vp_{\ep}$ on the whole $X$ by the first item of Proposition \ref{prop:est}. 

\noindent 
A simple yet useful observation is that because $D$ is transverse to the fiber, then any (approximate) conic metric on $X$ will restrict to a (approximate) conic metric on the fibers. Because of this, one can construct a reference metric that encodes the behavior of $\omte$, cf \cite[\S 3]{CGP}. Let us briefly recall the construction here. First, one introduces the function $\chi$ on $[\ep^2, +\infty)$ by 
$$\chi(\varepsilon^2+ t)= \frac 1 {\beta}\int_0^t \frac{(\varepsilon^2+ r)^{\beta}- \varepsilon^{2\beta}}{ r}dr$$
An easy computation shows that: 
$$\d \bar \d \big( \chi(\varepsilon^2+ |s_i|^2)\big)= \frac{\langle D^\prime s_i, D^\prime s_i\rangle}{(\varepsilon^2+ |s_i|^2)^{1-\beta_i}}- \frac{1}{\beta_i}
(\varepsilon^2+ |s_i|^2)^{\beta_i}- \varepsilon^{2\beta_i}) \, \theta_{D_i}$$
If we set $\chi_{i,\ep}:=  \chi(\varepsilon^2+ |s_i|^2)$ and $\chi_{\ep}=\sum_i \chi_{i,\ep}$, then the metric $\ombe:=\om+  \ddc \chi_{\ep}$ is a Kähler metric on $X$ (say when $\ep$ is small enough, and up to rescaling the hermitian metrics on $\Ox(D_i)$), and is uniformly quasi-isometric to 
\[\sum_{k=1}^r \frac{i \,dz_k\wedge d\bar z_{k}}{(|z_k|^{2}+\ep^2)^{1-\beta_k}} +\sum_{k=r+1}^n i \, dz_k\wedge d\bar z_k\]
near every point $p\in D$ where $D=(z_1\cdots z_r=0)$. Of course, $\ombte$ is (uniformly) quasi-isometric to $\om+ \sum_i \ddc (|s_i|^2+\ep^2)^{\beta_i}$, but it has better curvature properties (cf computations of \cite{CGP, GP}.
\noindent
For each fixed $t$, we know from \cite{GP} that $\vp_{t,\ep}$ converges uniformly to $\vp_t$ on $X_t$ (and smoothly outside $D_t$) and that $\omte$ is uniformly (in $\ep$) equivalent to $\ombte:={\ombe}_{|X_t}$. The following proposition shows uniformity in the variable $t$ of those results:

\begin{prop} 
\label{prop:est}
Up to shrinking $\DD$, there exists $C>0$ independent of $t$ and $\ep$ such that:
\begin{enumerate}
\item[$1.$] $||\vp_{t,\ep}||_{L^{\infty}(X_t)} \le C$
\item[$2.$] $C^{-1}\ombte \le \omte \le C \ombte$
\item[$3.$] $\Ric \omte \ge -C \omte$
\item[$4.$] The Sobolev and Poincaré constants of $\omte$ are uniformly bounded in $t,\ep$.
\end{enumerate}
\end{prop}

\begin{proof}
As $c_1(K_{X_t}+\sum(1-\beta_i)D_{i,t})+\{\gamma\}_{|X_t}$ is a Kähler class, there exists $\psi_t \in \mathscr C^{\infty}(X_t)$ such that $\Tht+\sum(1-\beta_i)\tt+\gamma_t + \ddc \psi_t$ is a Kähler form on $X_t$. Moreover, one can assume without loss of generality that $\psi_t$ induces a smooth function $\psi:x\mapsto \psi_{p(x)}(x)$. Therefore, up to shrinking $\DD$, the Kähler form ${\bar \om}_{\beta,t,\ep}:=\Tht+\sum_i(1-\beta_i)\tt+\gamma_t+ \ddc \psi_t+\ddc \chi_{\ep}$ is quasi-isometric to $\ombte$ uniformly with respect to $t,\ep$. Let now $u_{t,\ep}:=\vp_{t,\ep}-\psi_t-\chi_{\ep}$; this function on $X_t$ is solution of 
$$({\bar \om}_{\beta,t,\ep} + \ddc u_{t,\ep})^n = e^{u_{t,\ep}+\psi_t+\chi_{\ep}+f_{t,\ep}} \,{\bar \om}_{\beta,t,\ep}^n$$
where $f_{t,\ep}=\log\left(\om^n/\prod_i(|s_i|^2+\ep^2)^{1-\beta_i} {\bar \om}_{\beta,t,\ep}^n\right)$ is uniformly bounded in $t,\ep$. As $\psi_t, \chi_{\ep}, f_{t,\ep}$ are uniformly bounded, the maximum principle guarantees that $u_{t,\ep}$ is uniformly bounded, hence
$$||\vp_{t,\ep}||_{L^{\infty}(X_t)} \le C$$
for some $C>0$ independent of $t,\ep$. 

Now ${\bar \om}_{\beta,t,\ep}$ is the sum of a Kähler form (varying smoothly in $t$) and $\ddc \chi_{\ep}$; as $D$ is transverse to $X_t$, the local computations of \cite[\S 3]{GP} apply uniformly in $t$, and therefore the Laplacian estimate \cite[Proposition 1]{GP} holds as well: $$C^{-1}{\bar \om}_{\beta,t,\ep} \le \omte \le C{\bar \om}_{\beta,t,\ep}$$ As ${\bar \om}_{\beta,t,\ep}$ and $\ombte$ are uniformly quasi-isometric, we get 2. 

The previous estimate immediately a bound on the diameter of $X_t$ with respect to $\omte$:
\begin{equation}
\label{diam}
\mathrm{diam}(X_t, \omte) \le C
\end{equation}
as $\ombte$ is dominated by the restriction of a conic metric for $(X,D)$ on $X_t$, for which finiteness of the diameter can be checked easily. Moreover, recall from equation \eqref{eq:eps} that $\Ric \omte \ge - \omte + \sum_i \frac{\ep^2}{|s_i|^2+\ep^2} \, \tt$. From the estimate 2. above, we also deduce that there exists $C>0$ such that $\omte \ge C^{-1} \om_{|X_t}$. As a consequence of those two inequalities, we find that 
\begin{equation}
\label{ricci}
\Ric \omte \ge -C \omte
\end{equation}
for some uniform $C>0$. Combining equations \eqref{diam}-\eqref{ricci} with Yau's results \cite{Yau75} yields the control of the first eigenvalue of the Laplacian $\Delta_{\omte}$. Hence we get a uniform Poincaré inequality. As for the Sobolev one, if follows from \cite{Croke} given the control of the previous geometric quantities. 
\end{proof}

\subsection{Heat kernel estimates}
In the following, one works on a fixed fiber $X_t$, that one will rename $X$ to avoid confusion with the time variable. 
Let us define the operator $L_{\ep}:=-\Delta_{\omte}+\mathrm{Id}$ on the compact Kähler manifold $(X, \ome)$. It is well know that this operator is positive (that is, is $\Le f \ge 0$, then $f\ge 0$). More precisely, the unique solution $f$ to the equation $L_{\ep}f=g$ for a given smooth function $g$ on $X$ satisfies:
$$f(z)=\int_{0}^{+\infty}e^{-t}\int_X P_{\ep}(t,z,w) g(w) dV_{\ome}(w)$$ 
where $P_{\ep}(t,z,w)>0$ is the heat Kernel of $\frac{\d}{\d t}-\Delta_{\ome}$. From Proposition \ref{prop:est}, one gets a constant $k>0$ such that $\Ric \ome \ge -(2n-1)k \, \ome$; it follows then from \cite{CheegerYau} that:
$$P_{\ep}(t,z,w) \ge \frac{1}{(2\pi t)^n}e^{-\frac{d_{\ome}(z,w)^2}{t}}e^{-kt/4}$$
where $d_{\ome}(z,w)$ is the geodesic distance between $z$ and $w$ with respect to $\ome$. As the the diameter of $(X,\ome)$ is uniformly bounded, one gets 
\begin{equation}
\label{heat0}
\int_{0}^{+\infty}P_{\ep}(t,z,w)dt \ge c>0
\end{equation}
for some uniform $c>0$. In particular, if $g\ge 0$, then $\inf_X f \ge c \int_X g dV_{\ome}$.
This estimation is very useful in the case $D=0$, cf \cite[\S 3]{Schum} but it won't be sufficient for us because of the residual negative error terms due to our approximation.\\

Indeed, later on (cf eq. \eqref{infc} and the few lines below it) one will establish that $L_{\ep}c(\re)\ge -g_{\ep}$ where $g_{\ep}$ is a positive function whose $L^1$ norm tends to $0$ as $\ep\to 0$. One would like to conclude from there that $\liminf_{\ep \to 0}\inf_{X_t} c(\re) \ge 0$, but the lower bound on the heat kernel above won't be of any help. Rather, one should aim for an upper bound; it is provided by \cite[Theorem 16 \& 17]{Davies} who shows that for any $\delta>0$, there exists $c_{\delta}>0$ such that
$$P_{\ep}(t,z,w) \le \frac{c_{\delta}e^{\delta t}e^{-\frac{d_{\ome}(z,w)^2}{(4+\delta)t}}}{\vol(B_z(\sqrt t))^{1/2} \cdotp \vol(B_w(\sqrt t))^{1/2} } $$
where $\vol(B_z(\sqrt t))^{1/2} $ is the volume of the ball of radius $\sqrt{t}$ centered at $z$. Thanks to Proposition \ref{prop:est}, one can easily check that $\vol(B_z(t)) \ge c t^{2n}$ for some uniform $c>0$. Setting $r:=d_{\ome}(z,w)^2$, one gets:
\begin{equation}
\label{heat}
\int_{0}^{+\infty}e^{-t} P_{\ep}(t,z,w)dt \le C \int_{0}^{+\infty}e^{-\frac{r^2}{5t}}e^{-t/2}t^{-n}dt
\end{equation}
for some uniform $C>0$. After substituing $u:=r^2/t$, one gets
$r^{2n-2}\int_{0}^{+\infty}u^{n-2}e^{-u/5}e^{-r^2/2u}du$ which is dominated by $r^{2n-2}\int_{0}^{+\infty}u^{n-2}e^{-u/5}du$ if $n>1$ and by $\sqrt 2 \, r^{-1}\int_{0}^{+\infty}u^{-1/2}e^{-u/5}du$ if $n=1$ (using $e^{-x}\le x^{-1/2}$ for $x\ge 0$). To sum up, one has:
$$\int_{0}^{+\infty}e^{-t} P_{\ep}(t,z,w)dt \le 
\begin{cases}
\frac{C}{d_{\ome}(z,w)^{2n-2}} & \mbox{if } n>1 \\
\frac{C}{d_{\ome}(z,w)} & \mbox{if } n=1
\end{cases}$$
Eventually, one gets the following result:

\begin{prop}
\label{minoration}
Let $f,g$ be smooth functions on $X$ satisfying $\Le f=g$. Then there exist two constant $c,C>0$ independent of $\ep$ such that:
$$f(z)\ge c \int_X g_+dV_{\ome}-C \int_X\frac{g_{-}}{d_{\ome}(z,\cdotp)^{2n-2}}dV_{\ome}$$
if $n>1$ (replace the exponent $2n-2$ by $1$ if $n=1$. Here, $g_+ = \max\{g,0\}$, $g_{-}=\max\{-g,0\}$ and $d_{\ome}(z,w)$ denotes the distance between $z$ and $w$ with respect to the Kähler metric $\ome$ on $X$. 
\end{prop}
 
\begin{proof}
Recall that $P_{\ep}(t,z,w) \ge 0$; thanks to \eqref{heat0} and \eqref{heat}, one has:
 \begin{eqnarray*}
 f(z) &=& \int_{0}^{+\infty}e^t \left(\int_{\{g\ge 0\}} P_{\ep}(t,z,w) g(w) dV_{\ome}(w)+ \int_{\{g<0\}}P_{\ep}(t,z,w) g(w)\, dV_{\ome}(w)\right) dt \\
 & \ge & c \int_{\{g\ge 0\}} g(w) dV_{\ome}(w)- \int_{\{g< 0\}} \left(\int_{0}^{+\infty}e^{-t}P_{\ep}(t,z,w) dt \right)(-g(w)) dV_{\ome}(w)
 \end{eqnarray*}
 hence the result.
\end{proof}

\subsection{A gradient estimate}

We adapt B\l ocki's estimate in the context where a lower bound on the bisectional curvature is not available, but only a weaker bound as in \cite{GP}:

\begin{prop}
\label{prop:gradient}
Let $(X,\om)$ be a compact Kähler manifold, and let $\omvp:=\om+\ddc \vp$ be a Kähler metric satisfying $$\omvp^n= e^{\lambda \vp +F} \om^n$$ for some $F\in \mathscr C^{\infty}(X) $ and $\lambda\in \R$. We assume that there exists $C>0$ and a smooth function $\Psi$ such that:
\begin{enumerate}
\item[$(i)$] $\sup_X |\vp| \le C$
\item[$(ii)$] $\ddc \Psi \ge -C \om$ and $\, \sup_X |\Psi|\le C$
 \item[$(iii)$] $ \sup_X (|F|+|\nabla F|) \le C$
\item[$(iv)$] $i\Theta_{\om}(T_X) \ge -(C\om + \ddc \Psi) \otimes \mathrm{Id}$
\item[$(v)$] $\om' \ge C^{-1}\om$.
\end{enumerate}
Then there exists a constant $A>0$ depending only on $C$ and $n$ such that $|\nabla \vp |_{\om} \le C$.
\end{prop}

Condition $(v)$ is very constraining (it says that $\om'$ and $\om$ are already uniformly quasi-isometric), and should not be necessary (if $\Psi=0$, B\l ocki's result actually shows that we can discard it) although it does not seem obvious to us how to avoid assuming this condition in this setting of unbounded curvature. 

\begin{proof}
We follow very closely Blocki's proof. Let $\beta:= |\nabla \vp|^2$ (computed with respect to $\om$) and $\alpha:= \log \beta - \gamma \circ \vp$ where $\gamma $ is a function to specify later. Without loss of generality, one can assume $\inf \vp =0$, and we set $\sup \vp =: C_0$. We use the local notation $(g_{i\bar j})$ for $\om$. We work at a point $p\in X$ and choose a system of geodesic coordinates for $\om$ such that $g_{i\bar j}(p)=\delta_{i \bar j}$, $dg_{i\bar j}(p)=0$, and $\vp_{i\bar j}$ is diagonal.
We set $u_{i\bar j}=g_{i\bar j}+\vp_{i\bar j}$ the components of the metric $\om'$. Computations show that
$$\alpha_{p\bar p}=\frac{1}{\beta}\left(R_{j\bar k p \bar p} \vp_j \vp_{\bar k} +2 \mathrm{Re}\sum_ju_{p\bar p j}\vp_{\bar j}+\sum_j |\vp_{jp}|^2+\vp_{p\bar p}^2\right)-2\lambda-\left[(\gamma')^2+\gamma''\right]|\vp_p|^2-\gamma' \vp_{p\bar p}$$
so at $p$, the RHS is non-positive.

By the assumption $(iv)$, we have for all $a,b$: $R_{j\bar k p\bar q} a_j \bar a_{ k} b_p \bar b_{ q} \ge -(C|a_j|^2+\Psi_{j \bar k}a_j \bar a_{k})|b|^2$ and by symmetry of the curvature tensor, we get $R_{j \bar k p \bar q}a_j \bar a_{ k} b_p \bar b_{ q} \ge -(C|b_p|^2+\Psi_{p \bar q}b_p \bar b_{ q})|a|^2$. We apply that to $a= \nabla \vp$ and $b$ the vector with only non-zero component the $p$-th one, equal to $\sqrt{u^{p\bar p}}$, we get:
$u^{p\bar p}R_{j \bar k p \bar p}\vp_k \vp_{\bar l} \ge -(Cu^{p\bar p}+u^{p\bar p}\Psi_{p \bar p})|\nabla \vp|^2$. As a consequence, 
\begin{equation}
\label{courbure}
\frac{1}{\beta} \sum u^{p \bar p} R_{j \bar k p \bar p}\vp_j \vp_{\bar k} \ge -C\sum u^{p\bar p}-\sum_p u^{p\bar p}\Psi_{p\bar p}
\end{equation}
The next term to analyze is 
\begin{equation}
\label{gradient}
\frac{1}{\beta} \sum_p u^{p\bar p}\left( 2 \mathrm{Re}\sum_ju_{p\bar p j}\vp_{\bar j}\right)=\frac{2}{\beta}\mathrm{Re}\sum_j F_j \vp_{\bar j}
\end{equation}
by \cite[1.13]{Bl}, and this term is dominated (in norm) by $2 |\nabla F| \beta^{-1/2}$, hence by $2C \beta^{-1/2}$. Combining these last three relations, we get
\begin{equation}
\label{final}
\Delta' (\alpha+\Psi) \ge (\gamma'-C) \tr_{\om'}\om -2 ||\nabla F|| \beta^{-1/2}-\left[(\gamma')^2+\gamma''\right] |\nabla^{\om} \vp |^2_{\om'}-n \gamma'
\end{equation} 
We fix the function $\gamma$ on $[0,C_0]$ by setting $\gamma(t)= At-4(At)^2$ for $A=\frac{1}{C_0}$. It can be easily checked that on $[0,C_0]$, $(\gamma'^2)+\gamma'' \le -7 A^2$. 

Now we choose $p$ to be a point in $X$ where $\alpha+\Psi$ attains its maximum. Then either $\beta(p)<1$, or $\beta(p)\ge 1$. Assume for the time being that we are in the second case. As a consequence, $\beta^{-1/2}(p)\le 1$, and by $(v)$, equation \eqref{final} gives at $p$:
$$0 \ge \frac{7}{C_0^2} |\nabla^{\om} \vp |^2_{\om'} -C_1$$
for some bigger constant $C_1>0$.
By the equivalence of $\om$ and $\om'$, we get a new constant $C_2$ such that $\beta(p) \le \max(1, C_2)$. Now, if $x$ is any point in $X$, $\log \beta(x) = \alpha(x)+\gamma (\vp(x)) -\Psi(x) \le \alpha(p)+\sup \gamma +\sup |\Psi|\le \log \max(1,C_2)+2(\sup |\gamma|+\sup |\Psi|) \le C_3$. 
\end{proof}

\begin{coro}
\label{est:grad}
Let $\om_{t,\ep}$ the metric solution of Equation \eqref{eq:eps0}, whose potential is $\vp_{t,\ep}$, and let $\ombte$ the reference approximate conic metric from last paragraph. Then there exists $C>0$ independent of $\ep,t$ (up to shrinking $\mathbb D$) such that 
$$|\nabla \vp_{t,\ep} |\le C$$
where the gradient as well as its norm are computed with respect to $\ombte$.
\end{coro}

\begin{proof}
We need to check that the items $(i)-(v)$ from the above proposition are satisfied in this context. Proposition \ref{prop:est} gives $(i)$ and $(v)$. We get $(ii)$ and $(iv)$ (as well as the first bound for $(iii)$) from \cite[\S 4]{GP} where the appropriate function $\Psi$ is introduced. It remains to get a bound on the gradient of of $\log\left(\om^n/\prod_i(|s_i|^2+\ep^2)^{1-\beta_i} {\om}_{\beta,t,\ep}^n\right)$. Because the function in the logarithm, say $F_{\ep}$ is uniformly bounded away from $0$, it is enough to check that its gradient is bounded (all appropriate quantities are computed with respect to the approximate conic metric). From \cite[(21)]{CGP}, we extract that $F_{\ep}$ is (up to bounded terms in the usual $\mathscr C^1$ norm) a combination of $z_i (|z_i|^2+\ep^2)^{2\beta_i}$ and $(|z_i|^2+\ep^2)^{1+\beta_i}$. 
 Therefore, all the components of $\nabla F_{\ep}$ are (up to bounded) terms dominated by $|z_i|(|z_i|^2+\ep^2)^{\beta_i}+|z_i|(|z_i|^2+\ep^2)$ which is uniformly bounded. Therefore $|\nabla F_{\ep}|_{\rm eucl}$ is uniformly bounded, and as the approximate conic metric dominates a fixed Kähler metric, we get the expected result.
\end{proof}

The gradient $L^{\infty}$ estimate also leads to an $L^2$ estimate for the non-mixed second derivatives of $\vp_{t,\ep}$:

\begin{coro}
\label{est:hess}
There exists $C>0$ independent of $\ep, t$ such that: 
$$|| D^2 \vp_{t,\ep}||_{L^2} \le C$$
Here, $D$ and $L^2$ are taken with respect to $\omte$.
\end{coro} 

\begin{proof}
To lighten notation, we will use the following notations all along the proof: $\vp:=\vp_{t,\ep}$ and $\om:=\omte$. Bochner-Weitzenböck formula reads:
$$\Delta |\nabla \vp|^2 = |D^2 \vp|^2+\la \nabla \vp, \Delta \nabla \vp \ra + \Ric(\nabla \vp, \nabla \vp)$$
We know that $\Ric \om \ge -\om + \sum (1-\beta_i) \frac{\ep^2 \theta_{i,t}}{|s_i|^2+\ep^2}$ hence $\Ric \om \ge -C \om$ for some uniform $C>0$. Moreover, differentiating the equation 
$(\alpha_t+\ddc \vp)^n = \frac{e^{\vp}\om_0^n}{\prod (|s_i|^2+\ep^2)^{1-\beta_i}}$ (where $\alpha_t$ is the restriction to $X_t$ of $\alpha:= \Theta_{h^{\om_0}_{X/\DD}}(K_{X/\DD})+\sum_i(1-\beta_i) \theta_{D_i}$), we get 
$$\la\Delta \nabla \vp, \vp\ra = |\nabla \vp|^2+\Ric(\nabla \vp, \nabla \vp) - \sum (1-\beta_i)\la \nabla \log (|s_i|^2+\ep^2), \nabla \vp \ra +\la \tr_{\om_0}(D \om_0)-\tr_{\om}(D \alpha_t), \nabla \vp\ra$$
The term $\tr_{\om_0}(D \om_0)-\tr_{\om}(D \alpha_t)$ is obviously bounded, so that 
$$\Delta |\nabla \vp|^2 \ge  |D^2 \vp|^2 -C |\nabla \vp|^2-C |\nabla \vp | \,\left(1+ \sum_i |\nabla \log (|s_i|^2+\ep^2)|\right)$$
From Corollary \ref{est:grad} above, we know that $|\nabla \vp| \le C$ for some uniform $C$, and it is an easy computation to check that $||\nabla \log (|s_i|^2+\ep^2)||_{L^1} \le C$. As as result, $$|D^2 \vp|^2\le \Delta |\nabla \vp|^2+ F$$ with
$||F||_{L^1} \le C$. Therefore, $\int |D^2 \vp|^2 \le C$, which had to be proved. 
\end{proof}

\section{First order estimates in the transverse direction}

In this section, one first show that given any $\ep>0$, the potential $\vp_{t,\ep}$ varies smoothly with $t$. The main goal of the section is to derive a uniform estimate for $\frac{\d}{\d t} \vp_{t,\ep}$. For that purpose, one will consider suitable lifts of the vector field $\d/\d t$, and successively prove an $L^1$-like estimate, a $L^2$ one, and finally a $L^{\infty}$ estimate. \\

\subsection{Smoothness of the variation}

\begin{prop}
The function $\vpe$ is smooth on $X$.
\end{prop}

\begin{proof}
We work near the fiber $X_0$, and we can transpose the problem to a unique differential manifold $X_0$ with varying complex structure $J_t$ for $t$ close to the origin; we denote by $\partial_t, \bar \d_t$ the associated differential operators. The form $\Tht+\sum_i(1-\beta_i)\tt$, viewed on $X_0$, will be denoted by $\Theta_t$, and we write $\widetilde \omt$ for $\om_{|X_t}$ viewed on $X_0$. Up to shrinking $\DD$ one may assume that all the Kähler forms $\Theta_t$ on $X_0$ satisfy $\Theta_t \ge \frac 1 2 \Theta_0$; finally set $\mathcal P_0:= \mathrm{PSH}(X_0, \frac 1 2 \Theta_0)$. For each $\ep>0$, we define a map
$$ \begin{array}{ccccc}
\Phi_{\ep}&:&\R\times \mathscr C^{k+2,\alpha}(X) \cap \mathcal P_0 & \longrightarrow & \mathscr C^{k,\alpha}(X) \\
&&(t,u)&\longmapsto&\log\left[\frac{(\Theta_t+i \d_t \bar \d_t u)^n}{\widetilde \omega_t^n}\right]-u-\sum_i(1-\beta) \log(|s_{i,t}|^2+\ep^2)
\end{array}$$
Its partial derivative with respect to the function variable is
$$ \frac{\d \Phi_{\ep}}{\d u}(0, \vp_{0,\ep}) = \Delta_{\omega_{\vp_{0,\ep}}}-\mathrm {Id}$$
which is invertible. Hence there exists a unique smooth path $t\mapsto \vp_{t,\ep}$ such that $\Phi_{\ep}(t, \vp_{t,\ep})=0$ near the origin. Of course, modulo the identification of $X_t$ with $X_0$, this function is nothing else but the function $\vp_{t,\ep}$ introduced above, which legitimizes using the same notation for both functions. 
\end{proof}

\subsection{Integral bound}
To get the first bound, the following (standard) formula is crucial: 

\begin{prop}
\label{fint}
Let $\om$ be any Kähler form on $X$, let $f\in C^{\infty}(X)$, and let $v$ be any vector field on $X$ lifting $\frac{\d}{\d t}$. Then:
$$\frac{\d}{\d t}\left( \int_{X_t} f \om^n\right) \Big|_{t=t_0}= \int_{X_{t_0}} (v\, \cdotp f) \, \om^n$$ 
\end{prop}

\begin{proof}
Let $(\Phi_s)$ be the flow of $v$, so that $\Phi_s(X_0)=X_s$. Then $\int_{X_t} f \om^n=\int_{X_0}\Phi_t^*(f \om^n)$ so that $\frac{\d}{\d t} \int_{X_t} f \om^n=\int_{X_0}L_v(f \om^n)$ where $L_v$ is the Lie derivative of $v$. As $\om$ is closed, $L_v \om = di_v\om$, hence $L_v(f \om^n)=(v\cdotp f)\, \om^n+n L_v \om \wedge \om^{n-1}= (v\cdotp f)\, \om^n+d\left( n\, i_v \om \wedge \om^{n-1}\right)$ and the result follows from Stoke's formula.
\end{proof}

Let $\vep$ be the lift of $\frac{\d}{\d t}$ with respect to $\om_{\beta', t, \ep}$ for some $\beta'=(\beta'_1, \ldots, \beta'_N)$ with $\beta'_i \leq \min \{\beta_i, 1/2\}$. Locally, if one chooses coordinates $(z_1, \ldots, z_n, t)$ in $U$ such that $p(z_1, \ldots, z_n, t)=t$, then 
$$\vep = \frac{\d}{\d t}- \sum_{i,j}g^{\bar j i}g_{t \bar j} \frac{\d}{\d z_i}$$
if $\om_{\beta', t, \ep}= g_{t\bar t}\, idt \wedge d\bar t+i\sum_{k=1}^n(g_{k \bar t}dz_k \wedge d\bar t +g_{t \bar k}dz_t \wedge d \bar z_k  ) + i\sum_{k,l=1}^n g_{k \bar l}dz_k \wedge d \bar z_{l}$.
Recall from \cite[\S 4]{CGP} that one can choose the coordinates such that at the center $p_0\in X_0$ of the coordinate chart, the weights of the hermitian metrics on $D_i$ as well as their first derivatives vanish, so that we have
$$ g_{i\bar j } = \begin{cases}
O(|z_i|^2+\ep^2)^{\beta'_i-1}) & \mbox{if } i=j \in \{1, \ldots, r\} \\
O(1) & \mbox{else}
\end{cases}$$
and 
\begin{equation}
\label{inv}
 g^{i\bar j } = \begin{cases}
O(|z_i|^2+\ep^2)^{1-\beta'_i}) & \mbox{if } i=j \in \{1, \ldots, r\} \\
O\left(|z_i|^2+\ep^2)^{1-\beta'_i}(|z_j|^2+\ep^2)^{1-\beta'_j}\right) & \mbox{if } i,j \in \{1, \ldots, r\} \mbox{ and } i\neq j \\
O(1) & \mbox{else}
\end{cases}
\end{equation}
In particular, we have at $p_0$:
$$v_{\ep}=  \sum_{i=1}^r O\left(|z_i|^2+\ep^2)^{1-\beta'_i}\right) \,\cdotp \frac{\d}{\d z_i}+\sum_{i=r+1}^{n+1} O(1) \, \cdotp \frac{\d}{\d z_i}$$
At $p_0$, the following holds:  
$$\frac{\d}{\d z_k} \log (|s_i|^2+\ep^2) = 
\begin{cases} 
\frac{\bar z_i}{|z_i|^2+\ep^2} & \mbox{if } k=i \\
0 & \mbox{else}
\end{cases}$$
so that eventually, 
$(v_{\ep}\cdotp \log (|s_i|^2+\ep^2))(p_0)= O(|z_i|^2+\ep^2)^{\frac 12-\beta'_i}$
hence there exists a constant $C$ independent of $\ep$ such that:
\begin{equation}
\label{vep}
\left| v_{\ep}\cdotp \log (|s_i|^2+\ep^2) \right| \le C
\end{equation}
on $X_0$ (and actually this would hold uniformly on $X_t$ for a small $t$). If we piece these observations together, we get:

\begin{prop}
\label{L1}
There exists a constant $C>0$ such that 
$$\left| \int_{X_t} (\vep \cdotp \vp_{t,\ep}) \, \omte^n \right| \le C$$
\end{prop}

\begin{proof}
Recall the Monge-Ampère satisfied by $\om_{t,\ep}$ on $X_t$: $\omte^n= \frac{e^{\vp_{t,\ep}}}{\prod (|s_i|^2+\ep^2)^{1-\beta_i}}\om^n$. 
As $\omte$ lives in the cohomology class of $K_{X_t}+(1-\beta) D_t$, its volume is constant, hence $\frac{\d}{\d t} \int_{X_t}\frac{e^{\vp_{t,\ep}}}{\prod (|s_i|^2+\ep^2)^{1-\beta_i}}\om^n=0$. By  Proposition \ref{fint}, we get:
\begin{equation}
\label{div}
\int_{X_t} (\vep \cdotp \vp_{t,\ep}) \, \omte^n =   \sum_i (1-\beta_i) \int_{X_t} (\vep \cdotp \log(|s_i|^2+\ep^2)) \, \omte^n
\end{equation}
So we are left to showing that the right hand side is uniformly bounded, but this is a consequence of \eqref{vep}.
\end{proof}

\subsection{$L^2$ bounds}

We are interested in estimating $\frac{\d}{\d t}\vp_{t,\ep}$ (which is only locally defined); although this function satisfies a very simple equation (essentially $\Delta-\mathrm{Id}$ of this function is uniformly bounded), local methods don't seem to easily provide a bound for it. Instead, we work globally on $X_t$, and estimate $v_{\ep} \cdotp \vp_{t, \ep}$, where $v_{\ep}$ is the vector field introduced in the previous section. 
We are gaining compactness (so no boundary terms in the integrations by parts), but it involves differentiating with respect to the "conic directions" which creates singular terms. The goal of this section is to prove:

\begin{prop}
\label{C0}
There exists a constant $C>0$ independent of $\ep,t$ (chosen small enough) such that
$$\left| \!\left| v_{\ep} \cdotp \vp_{t, \ep}\right| \! \right|_{L^{2}(X_t)} \le C$$ 
\end{prop}

\vspace{3mm}

\noindent
The strategy of the proof is simple: differentiate the Monge-Ampère equation satisfied by $\vp_{t,\ep}$ with respect to $v_{\ep}$ to obtain an elliptic linear equation satisfied by $v_{\ep} \cdotp \vp_{t,\ep}$ and apply the standard arguments. The difficulty will actually consist of analyzing precisely the coefficients of the linear equation. Let us give some more details now. 

\noindent
Remember that $\vp_{t,\ep}$ solves the equation
$$(\alpha_t+\ddc \vp_{t,\ep})^n = \frac{e^{\vp_{t,\ep}}\om^n}{\prod(|s_i|^2+\ep^2)^{1-\beta_i}}$$
where $\alpha_t= \Tht+\sum(1-\beta_i)\tt+\gamma_t$ (it is the restriction to $X_t$ of the obvious (smooth) form $\alpha$ on $X$).
Differentiating this equation, we get:
$$\tr_{\omte}(\vep \cdotp(\alpha_t+\ddc \vp_{t,\ep})) = \vep \cdotp\vp_{t,\ep}+\tr_{\om}(\vep \cdotp \om)- \sum_{i=1}^N (1-\beta_i) \vep \cdotp \log(|s_i|^2+\ep^2)$$
We would like to consider this equation as an elliptic PDE satisfied by $\vep \cdotp \vp_{t,\ep}$; however, as $\vep$ is not holomorphic, it does not commute with the Laplace operator. More precisely, we have
\begin{equation}
\label{eq:lin0}
\left(\Delta_{\omte}-1\right) (\vep \cdotp \vp_{t,\ep})=R_1+R_2
\end{equation}
where $R_1= \tr_{\om}(\vep \cdotp \om)-\tr_{\omte}(\vep \cdotp \alpha_t)- \sum_{i=1}^N (1-\beta_i)\, \vep \cdotp \log(|s_i|^2+\ep^2)$ and $R_2= \tr_{\omte}(\vep \cdotp \ddc \vp_{t,\ep}) - \tr_{\omte}(\ddc(\vep \cdotp \vp_{t,\ep})) $. Here,  $R_2$ measures the non holomorphicity of $\vep$ in the fiber directions. We claim that if the angles $\beta'_j$ are suitably chosen, the right hand side of \eqref{eq:lin0} is admits a uniform $L^2$ bound, which we prove in the following two steps.\\

\begin{lemm}
\label{lem:L2}
There exists a uniform constant $C>0$ such that
\begin{enumerate}
\item[$(i)$] $||R_1||_{L^{\infty}} \le C$ ;
\item[$(ii)$] $||R_2||_{L^{2}} \le C$.
\end{enumerate}
\end{lemm}

\begin{proof}
The first part is easy at this point. Indeed, $\vep \cdotp \om$ and $\vep \cdotp \alpha_t$ have uniformly bounded coefficients, so their trace with respect to $\om$ or $\omte$ are uniformly bounded. The remaining terms to bound are $ \vep \cdotp \log(|s_i|^2+\ep^2)$, but we analyzed them already, cf \eqref{vep}.\\

The second item is more involved. To lighten notation, we will set $\vp:=\vp_{t,\ep}$ and $\om:=\omte$ all along the proof of the lemma. In the usual chosen coordinates, we write $\vep=\frac{\d}{\d t}- \sum_{k=1}^n v_k \d_k$, with $v_k = \sum_l h^{\bar l k} h_{t \bar l}$ for $h$ the approximate conic metric with cone angles $\beta_i'$ along $D_i$. An elementary computation shows that:
$$R_2=\sum_k \Delta v_k \cdotp \d_k \vp+\sum_{i,j,k}\om^{i\bar j} \d_i v_k \cdotp \d_{\bar j k }\vp+\sum_{i,j,k}\om^{i\bar j} \d_{\bar j} v_k \cdotp \d_{i k }\vp$$

Let us finally recall the following estimates (holding at $p_0$ the center of the coordinate chart) for the derivatives of the coefficients of the approximate conic metric, extracted from \cite[\S 4.3.2]{CGP}:
{\footnotesize
\begin{equation}
\label{deriv}
g_{i\bar j, k}= \begin{cases}
O(\bar z_i (|z_i|^2+\ep^2)^{\beta_i-2}) &  \mbox{   if } i=j=k \in \{1, \ldots, r\}\\
O(1+\delta_i \bar z_i (|z_i|^2+\ep^2)^{\beta_i-1}+\delta_j \bar z_j (|z_j|^2+\ep^2)^{\beta_j-1}+\delta_k  \bar z_k (|z_k|^2+\ep^2)^{\beta_k-1} )&  \mbox{ else}
\end{cases}
\end{equation}}
where $\delta_i = 1$ if $i\in \{1, \ldots, r\}$ and $0$ otherwise.\\

Going back to $R_2$, one needs to estimate three terms. 
\begin{enumerate}
\item[$a.$] The term $\Delta v_k \cdotp \d_k \vp$.\\

\noindent
Remember from the gradient estimate that
$$|\d_k \vp|^2 = \begin{cases}
O((|z_k|^2+\ep^2)^{\beta_k -1}) & \mbox{if } k\in\{1,\ldots,r\}  \\
O(1) & \mbox{else}
\end{cases} $$
As $v_k=\sum_l h^{\bar l k}h_{t \bar l}$, we have: $$\d_{i\bar j}v_k=(h^{\bar l k})_{i\bar j}h_{t\bar l}+(h^{\bar l k})_i h_{t\bar l, \bar j}+(h^{\bar l k})_{\bar j}h_{t\bar l,i}+h^{\bar l k}h_{t\bar l, i \bar j}$$ Let us estimate all the summands involved. First, {\footnotesize
$$(h^{\bar l k})_i = -h^{\bar l \alpha}h_{\bar \alpha \beta, i}h^{\bar \beta k}=\begin{cases}
O(|z_i|(|z_i|^2+\ep^2)^{-\beta'_i}) & \mbox{if } i,k,l\in \{1, \ldots, r\}, i=k=l\\
O(|z_i|(|z_i|^2+\ep^2)^{-\beta'_i}(|z_l|^2+\ep^2)^{1-\beta'_l}) & \mbox{if } i,k,l\in \{1, \ldots, r\}, i=k,l\neq i\\
O(|z_i|(|z_i|^2+\ep^2)^{-\beta'_i}(|z_k|^2+\ep^2)^{1-\beta'_k}) & \mbox{if } i,k,l\in \{1, \ldots, r\}, i=l,k\neq i\\
O(|z_i|(|z_i|^2+\ep^2)^{\beta'_i-1}(|z_k|^2+\ep^2)^{1-\beta'_k}(|z_l|^2+\ep^2)^{1-\beta_l})) & \mbox{if } i,k,l\in \{1, \ldots, r\}, k,l\neq i\\
O(|z_i|(|z_i|^2+\ep^2)^{\beta_i'-1}) & \mbox{if } i\in \{1, \ldots, r\}, l,k \notin \{1, \ldots, r\}\\
O((|z_k|^2+\ep^2)^{1-\beta'_k}(|z_l|^2+\ep^2)^{1-\beta'_l}) & \mbox{if } i\notin \{1, \ldots, r\}, l,k \in \{1, \ldots, r\}\\
O((|z_k|^2+\ep^2)^{1-\beta'_k} & \mbox{if } i,l\notin \{1, \ldots, r\}, k \in \{1, \ldots, r\}\\
O((|z_l|^2+\ep^2)^{1-\beta'_l})) & \mbox{if } i,k\notin \{1, \ldots, r\}, l \in \{1, \ldots, r\}\\
O(1) & \mbox{else}
\end{cases}$$}
For the second derivatives, we write $$-(h^{\bar l k})_{i \bar j}=(h^{\bar l \alpha})_{\bar j}h_{\bar \alpha \beta, i}h^{\bar \beta k}+h^{\bar l \alpha}h_{\bar \alpha \beta, i\bar j}h^{\bar \beta k}+h^{\bar l \alpha}h_{\bar \alpha \beta, i}(h^{\bar \beta k})_{\bar j}$$ and use the fact (cf \cite[Eq. (23)]{CGP}) that, at the center of the chart: 
$h_{\alpha\bar \beta, i\bar j}=O((|z_i|^2+\ep^2)^{\beta'_i-2})$ if all four indexes $\alpha, \beta, i,j$ are equal and belong to $\{1, \ldots r\}$ and else, up to bounded terms, its expansion involves terms of the form $z_s(|z_s|^2+\ep^2)^{\beta_s'-1}$ and $\delta_s (|z_s|^2+\ep^2)^{\beta_s'-1}$ for $s\in \{\alpha, \beta, i,j\}\cap \{1, \ldots, r\}$ and where $\delta_s\in \{0,1\}$ is equal to $1$ iff $s$ appears at least twice in $(\alpha, \beta, i,j)$. We deduce from this:
{\footnotesize
$$(h^{\bar l k})_{i\bar j}=\begin{cases}
O((|z_i|^2+\ep^2)^{-\beta'_i}) & \mbox{if } i,j \in \{1, \ldots , r\}, i=j=k\\
O((|z_i|^2+\ep^2)^{\beta'_i-1}(|z_k|^2+\ep^2)^{\delta_k(1-\beta'_k)})& \mbox{if } i,j\in \{1, \ldots, r\},i=j, k\neq i\\
O(z_i(|z_i|^2+\ep^2)^{\beta'_i-1}z_j (|z_j|^2+\ep^2)^{\beta'_j-1}  (|z_k|^2+\ep^2)^{\delta_k(1-\beta'_k)})& \mbox{if } i,j\in \{1, \ldots, r\}, i\neq j \\
O(z_i(|z_i|^2+\ep^2)^{\beta'_i-1}  (|z_k|^2+\ep^2)^{\delta_k(1-\beta'_k)})& \mbox{if } i\in \{1, \ldots, r\},  j \notin  \{1, \ldots, r\} \\
O(z_j (|z_j|^2+\ep^2)^{\beta'_j-1}  (|z_k|^2+\ep^2)^{\delta_k(1-\beta'_k)})& \mbox{if } i\notin \{1, \ldots, r\},  j \in  \{1, \ldots, r\} \\
O(1) & \mbox{else}
\end{cases}$$}
where $\delta_k=1$ if $k\in \{1, \ldots, r\}$ and $0$ else. One can easily see that the previous term is the most singular in the expansion of $\d_{i \bar j}v_k$, so in conclusion:
Combining all these estimates, we obtain finally 
{\footnotesize
$$|\om^{i \bar j}\d_{i \bar j}v_k|^2 = \begin{cases}
O((|z_i|^2+\ep^2)^{2-2\beta_j-2\beta'_i}) & \mbox{if } i,j,k\in \{1, \ldots, r\}, i=j=k \\
O((|z_i|^2+\ep^2)^{2\beta'_i-2\beta_i}(|z_k|^2+\ep^2)^{2\delta_k(1-\beta'_k)})& \mbox{if } i,j\in \{1, \ldots, r\},i=j, k\neq i\\
O(|z_i|^2(|z_i|^2+\ep^2)^{2\beta'_i-\beta_i-1}|z_j|^2 (|z_j|^2+\ep^2)^{2\beta'_j-\beta_j-1}  (|z_k|^2+\ep^2)^{2\delta_k(1-\beta'_k)})& \mbox{if } i,j\in \{1, \ldots, r\}, i\neq j \\
O(|z_i|^2(|z_i|^2+\ep^2)^{2\beta'_i-\beta_i-1}  (|z_k|^2+\ep^2)^{2\delta_k(1-\beta'_k)})& \mbox{if } i\in \{1, \ldots, r\},  j \notin  \{1, \ldots, r\} \\
O(|z_j|^2 (|z_j|^2+\ep^2)^{2\beta'_j-\beta_j-1}   (|z_k|^2+\ep^2)^{2\delta_k(1-\beta'_k)})& \mbox{if } i\notin \{1, \ldots, r\},  j \in  \{1, \ldots, r\} \\
O(1) & \mbox{else}
\end{cases}$$}
Assume first that $k\in \{1, \ldots, r\}$. Then it follows from the estimates above that $$|\Delta v_k|^2 |\d_k \vp|^2 =O\left((|z_k|^2+\ep^2)^{1-\beta_k-2\beta'_k}+ \sum_{i\neq k}(|z_i|^2+\ep^2)^{2\beta_i'-\beta_i}\right)$$
Moreover, we have $1-\beta_k-2\beta'_k+(\beta_k-1)>-1$ and $2\beta_i'-\beta_i+(\beta_i-1)>-1$ as long as the angles $\beta'$ satisfy $\beta_i'<1/2$ for all $i$, condition which then guarantees that $||\Delta v_k \cdotp \d_k \vp||_{L^2}\le C$. \\

\item[$b.$] The term $\sum_{i,j}\om^{i\bar j} \d_i v_k \cdotp \d_{\bar j k }\vp$. 

\noindent
Given that $\om_{\bar j k}= (\alpha_t)_{\bar j k} + \d_{\bar j k }\vp$, that term splits as $\d_k v_k-\sum_{i,j}\om^{i\bar j} \d_i v_k \cdotp  (\alpha_t)_{\bar j k} $. But the estimates provided in \eqref{vder}, show that $||\d_i v_k||_{L^2} \le C$ for any $i,k$.\\

\item[$c.$] The term $\sum_{i,j,k}\om^{i\bar j} \d_{\bar j} v_k \cdotp \d_{i k }\vp$.

\noindent
From Corollary \ref{est:hess}, we deduce the existence of a positive function $H$ (depending on $\ep$) such that 
$\int H \om^n \le C$ and:
$$H^{-1}|\d_{i k }\vp|^2 = \begin{cases}
O((|z_i|^2+\ep^2)^{\beta_i-1}(|z_k|^2+\ep^2)^{\beta_k-1})& \mbox{if } i,k\in \{1, \ldots, r\}\\
O((|z_i|^2+\ep^2)^{\beta_i-1}) & \mbox{if } i\in \{1, \ldots, r\}, k\notin  \{1, \ldots, r\} \\
O((|z_k|^2+\ep^2)^{\beta_k-1}) & \mbox{if } i\notin \{1, \ldots, r\}, k\in  \{1, \ldots, r\} \\
O(1) & \mbox{else}
\end{cases}$$
Combining this information with \eqref{vder}, a tedious but straightforward case study shows that
{\footnotesize
$$ H^{-1}|\om^{i\bar j} \d_{\bar j} v_k \cdotp \d_{i k }\vp|^2 = \begin{cases}
O(|z_k|^2 (|z_k|^2+\ep^2)^{\beta_k-1}  |z_j|^2(|z_j|^2+\ep^2)^{2\beta_j'-\beta_j-1} ) & \mbox{if } j,k \in \{1, \ldots, r\}, j\neq k \\
O(|z_k|^2 (|z_k|^2+\ep^2)^{\beta_k-1} )& \mbox{if } k \in \{1, \ldots, r\}, j\notin  \{1, \ldots ,r \} \\
O(|z_j|^2(|z_j|^2+\ep^2)^{2\beta_j'-\beta_j-1} ) & \mbox{if } k \notin \{1, \ldots, r\}, j\in  \{1, \ldots ,r \} \\
O(|z_k|^2)& \mbox{if } j,k \in \{1, \ldots, r\}, j= k \\
O(1) & \mbox{else}
\end{cases}$$}
and therefore $|\om^{i\bar j} \d_{\bar j} v_k \cdotp \d_{i k }\vp|^2=O((|z_j|^2+\ep^2)^{2\beta_j'-\beta_j}H)$ which is uniformly integrable with respect to $\om$ as long as $2\beta'_j\ge \beta_j$ for all $j$. So if we choose $\beta_j'\in [\beta_j/2, 1/2)\neq \emptyset$, we will have $\beta_j' < 1/2$ and $2 \beta'_j \ge \beta_j$ which are the two conditions needed on $\beta'_j$ so far. 
\end{enumerate}
This concludes the proof of the lemma. 
\end{proof}

In summary, we have proved that $||R_1||_{L^2}+||R_2||_{L^2}\le C$, hence $\vep \cdotp \vp_{t,\ep}$ satisfies an elliptic equation where the rhs has a uniform $L^2$ bound. One can now easily conclude the proof of the proposition. 

\begin{proof}[End of the proof of Proposition \ref{C0}]
One has to prove the estimate for the real part and the imaginary part of $\vep \cdotp\vp_{t,\ep}$. Because the two proofs are completely analogous, we will focus on the real part say. Also, we will drop the indexes $\ep,t$ to alleviate the notations, and we will set $u:=\mathrm{Re}(v_{\ep} \cdotp \vp_{t,\ep})$, $R:=\mathrm{Re}(R_1+R_2)$, and we will denote $\Delta:=\Delta_{\om_{t,\ep}}$ and work with $L^p$ spaces induced by the measure $\omte^n$. With these notations, Equation \eqref{eq:lin0} translates into
\begin{equation}
\label{eq:lin3}
\Delta u = u+R
\end{equation}
If one multiplies this equation by $u$ and integrate by parts, we get
$$\int_{X} |\nabla u|^2  +\int_{X}u^2 =-\int_{X}R u$$ and therefore
\begin{equation}
\label{ineq1}
||\nabla u ||_{L^2}^2 \le ||R||_{L^2} \cdotp ||u||_{L^2}
\end{equation}
From Proposition \ref{prop:est}, we know that $\omte$ has a uniform Poincaré constant $C_P$, so $||u-\int_{X_t}u ||_{L^2} \le C_P ||\nabla u||_{L^2}$. Combining this with \eqref{ineq1}, we obtain:
\begin{equation}
\label{ineq2}
||u||_{L^2} \le V \left |\int_{X_t} u  \right|+ C_P||R||^{1/2}_{L^2} \cdotp ||u||^{1/2}_{L^2}
\end{equation}
where $V=\int_{X_t} \omte^n$ is a constant. 


\noindent
One saw above that  $||R||_{L^2}\le C$. Combining this with Proposition \ref{L1}, \eqref{ineq2} becomes:
$$||u||_{L^2} \le C(1+||u||^{1/2}_{L^2})$$
and therefore $||u||_{L^2}\le C'$ (with $C'$ such that $C'^2=(C+\sqrt{C^2+4C})/2$)).
\end{proof}

\subsection{$L^{\infty}$ bounds}
\label{Linfty}
We obtained $L^2$ bounds for $\vep \cdotp \vp_{t,\ep}$ in the previous section. 
Recall that $\vep \cdotp \vp_{t,\ep}$ is solution of \eqref{eq:lin0}, whose right hand side $R_{\ep}$ is uniformly bounded for the $L^{2}$ norm only, and a priori not for the $L^{\infty}$ norm. This prevents from using the standard Harnack inequality to get $L^{\infty}$ estimates for the solution. However, because we now have $L^2$ estimates, provided by global methods, one can go back to the local equation satisfied by $\d_t \vp_{t,\ep}$ (say on a trivializing chart $U_t\subset X_t$) which is drastically simpler: 
$$(\Delta_{\omte}-1)(\d_t \vp_{t,\ep})= \tr_{\om}(\d_t \om)-\tr_{\omte}(\d_t \alpha_t)- \sum_{i=1}^N (1-\beta_i) \d_t \log(|s_i|^2+\ep^2)$$
Obviously, the right hand side of this equation has uniform $L^{\infty}$ bounds. Moreover, $\d_t \vp_{t,\ep}=\vep \cdotp \vp_{t,\ep}+O(1)$ given the gradient estimate (Corollary \ref{est:grad}), so $||\d_t \vp||_{L^2}\le C$. Now one can use the Harnack inequality (see e.g. \cite[Theorem 8.17]{Gilb}) to get:
\begin{equation*}
||\d_t \vp_{t,\ep}||_{L^{\infty}(U_t)} \le C
\end{equation*}
and ultimately
\begin{equation}
\label{Linf}
||v_{\ep}\cdotp \vp_{t,\ep}||_{L^{\infty}(X_t)} \le C
\end{equation}
Implicitly, we used that Harnack inequality works just as well for $\omte$ because it satisfies uniform Poincaré and Sobolev inequalities (and being closed, it is legitimate to integrate by parts with respect to that form). 

Now, let $u$ denote either the real part or the imaginary part of $\d_t\vp_{t,\ep}$; one has $\Delta u-u=\mathrm{Re}\left((\tr_{\omte}-\tr_{\om})(\d_t \om)-(1-\beta) \,\d_t \log(|s|^2+\ep^2)\right)$ (or similarly with $\mathrm{Im}$). Outside $D$ the right hand side of this equation has uniform $\mathscr C^k$ bounds for all $k$ (and so does $\omte$). Given the estimate \eqref{Linf}, one can apply Schauder estimates to get
\begin{eqnarray}
\label{Linf2}
\left|\!\left|\frac{\d \vp_{t,\ep}}{\d t}\right|\!\right|_{\mathscr C^{k}(U)} \le C_{U,k}
\end{eqnarray}
for each relatively compact open subset $U \Subset X_t \smallsetminus D_t$, uniformly in $t$ (small).

\section{Positivity of the variation}
\label{positivity}

One considers the vector field $\vph$ defined as the lift of $\frac{\d}{\d t}$ with respect to the approximate Kähler-Einstein metric $\omte$. Remember that $\om$ is a background fixed Kähler form (living on the total space). One will prove:

\begin{prop}
\label{prop:L2}
There exists a constant $C>0$ independent of $t,\ep$ such that:
$$\int_{X_t} |\vph|^2_{\om}\,  \omte^n \le C$$ 
\end{prop} 

\begin{proof}
This statement can be checked locally, so we choose the usual system of coordinates, around a point $0\in X_t$. We write $(g_{i\bar j})$ (resp. $(h_{i\bar j})$) for the components of $\om_{\beta',t,\ep}$ (resp. $\omte$) in these coordinates. Recall that $\vph = \frac{\d}{\d t}- \sum_{i,j}h^{\bar j i}h_{t \bar j} \frac{\d}{\d z_i}$.  As $\omte$ and $\ombte$ are uniformly quasi-isometric on $X_t$, we will work on the Kähler manifold $(X_t, \ombte)$ in the following (so all gradients and $L^p$ norms will be considered with respect to $\ombte$). Once again, we will drop the indexes $\ep,t$ to lighten notation (so $v:=\vep, \vp:=\vp_{t,\ep}$). Note that up to a harmless term, $h_{i \bar j}=\frac{\d\vp}{dz_i d \bar z_j}$.

 The key observation is that the $L^{2}$ bound on $v \cdotp \vp$ obtained in Proposition \ref{C0} provides a $L^{2}$ estimate on $\nabla (v \cdotp \vp)$ because of \eqref{ineq1} (and the bound on $||R_{\ep}||_{L^2}$ proved a few lines below that inequality):
 \begin{equation}
 \label{L2}
 \int_{X_t} |\nabla (v \cdotp \vp)|^2 \, \ombte^n \le C
 \end{equation}
We decompose the gradient of a function $f$ as $\nabla f = \sum_i \nabla^i f \frac{\d}{\d \bar z_i}  $, or equivalently $\nabla^i f = \sum_j g^{i\bar j} \frac{\d}{\d \bar z_j}$. With respect to the background metric $\om$, $|\vph|^2_{\om}$ is controlled by $\sum_i \left | \sum_j h^{\bar j i}h_{t\bar j}\right|^2 \simeq \sum_i \left | \sum_j h^{\bar j i}\frac{\d}{\d z\bar j}\left(\frac{\d \vp}{\d t}\right)\right|^2$ which is $\left|\nabla^{\omte}\left( \d_t \vp \right)\right|^2$, itself controlled by $|\nabla(\d_t \vp)|^2$. So we are left to estimate $||\nabla(\d_t \vp)||_{L^2}$.

We want to relate that last quantity to $||\nabla (v \cdotp \vp)||_{L^2}$, which we have under control by \eqref{L2}. Given the definition of $v$, we get $\nabla (\d_t \vp)= \nabla(v \cdotp \vp) +\nabla(g^{\bar j k }g_{t\bar j}\d_k \vp)$, hence everything comes down to bounding the $L^2$ norm of the second term on the rhs. Given $i\in \{1, \ldots, n\}$, we compute
\begin{equation}
\label{nabla}
\nabla^i(g^{\bar j k }g_{t\bar j}\d_k \vp)=g_{t\bar j}g^{\bar j k}g^{i \bar l} (\d_{k \bar l}\vp)+(\nabla^ig_{t\bar j}) g^{\bar j k}\d_k \vp+ (\nabla^i g^{\bar j k}) g_{t \bar j}\d_k \vp
\end{equation}
We are going to bound each of the three terms in \eqref{nabla} above, say at $0$. Before that, observe that the bound $|\nabla \vp| \le C$ provided by Corollary \ref{est:grad} combined with the fact that $\ombte$ is uniformly quasi-isometric to the model $\sum_{k=1}^r (|z_k|^2+\ep^2)^{\beta_k-1} dz_k\wedge d\bar z_k + \sum_{k\ge r+1} dz_k \wedge d\bar z_k$ yields 
$$|\d_k \vp|^2 = \begin{cases}
O((|z_k|^2+\ep^2)^{\beta_k -1}) & \mbox{if } k\in\{1,\ldots,r\}  \\
O(1) & \mbox{else}
\end{cases} $$
We can start estimates the three terms of \eqref{nabla} now.

\begin{enumerate}
\item[$a.$] The term $g_{t\bar j}g^{\bar j k}g^{i \bar l} (\d_{k \bar l}\vp)$.

\noindent 
As $\ombte$ and $\omte$ are quasi-isometric, we see easily that the first term is bounded, given that $|g_{t\bar j}(0)| \le C$.\\

\item[$b.$] The term $(\nabla^ig_{t\bar j}) g^{\bar j k}\d_k \vp$.

\noindent
By the estimate on $\d_k \vp$ above, we get that for each $j$, $|g^{\bar j k}\d_k \vp| \le C$. Moreover, $\nabla^i g_{t\bar j}= \sum_k g^{\bar k i} \d_{\bar k} g_{t\bar j}$; combining the estimates \eqref{inv} and \eqref{deriv}, we see that $g^{\bar k i} \d_{\bar k} g_{t\bar j}$ is always uniformly bounded, hence so is $(\nabla^ig_{t\bar j}) g^{\bar j k}\d_k \vp$.\\

\item[$c.$] The term $(\nabla^i g^{\bar j k}) g_{t \bar j}\d_k \vp$.

\noindent
We have: $\nabla^i g^{\bar j k}=\sum_l g^{\bar l i}( \d_{\bar l}g^{\bar j k})= -\sum_{l,\alpha, \beta}g^{\bar l i}g^{\bar j \alpha} (\d_{\bar l}g_{\alpha \bar \beta}) g^{\bar \beta k}$. If $k\notin \{1, \ldots, r\}$, then
$$g^{\bar j \alpha}(\d_{\bar l}g_{\alpha \bar \beta})g^{\bar \beta k} = 
\begin{cases}
O(z_l (|z_l|^2+\ep^2)^{-\beta'_l}) & \mbox{if } \alpha=\beta=l\in\{1, \ldots, r\}\\
O(z_l (|z_l|^2+\ep^2)^{\beta'_l-1}) & \mbox{if }  (\alpha \mbox{ or } \beta \notin\{1, \ldots, r\}), l\in \{1, \ldots, r\}\\
O(1) & \mbox{else}
\end{cases}$$
and if $k\in\{1, \ldots, r\}$, 
$$g^{\bar l i}g^{\bar j \alpha}(\d_{\bar l}g_{\alpha \bar \beta})g^{\bar \beta k} = 
\begin{cases}
O(z_l (|z_l|^2+\ep^2)^{1-2\beta'_l} ) & \mbox{if } \alpha,\beta,l\in \{1, \ldots, r\}, k = l\\
O(z_l (|z_l|^2+\ep^2)^{1-2\beta'_l}(|z_k|^2+\ep^2)^{1-\beta'_k} ) & \mbox{if } \alpha,\beta,l\in \{1, \ldots, r\}, k \neq l\\
O(z_l (|z_k|^2+\ep^2)^{1-\beta'_k} ) & \mbox{if } (\alpha \mbox{ or } \beta \notin\{1, \ldots, r\}),l\in \{1, \ldots, r\}\\
O((|z_k|^2+\ep^2)^{1-\beta'_k}) &\mbox{else}

\end{cases}$$

\noindent
If $k\notin  \{1, \ldots, r\} $, then $g_{t \bar j}\d_k \vp$ is bounded, and $$|g^{\bar l i}g^{\bar j \alpha} (\d_{\bar l}g_{\alpha \bar \beta}) g^{\bar \beta k} g_{t\bar j}\d_k \vp|^2=
\begin{cases}
O(|z_l|^2 (|z_l|^2+\ep^2)^{2-4\beta'_l}+|z_l|^2) &\mbox{if } l \in  \{1, \ldots, r\} \\
O(1) & \mbox{else}
\end{cases}$$
If $k\in  \{1, \ldots, r\}$, then $|\d_k \vp|^2=O((|z_k|^2+\ep^2)^{\beta_k-1})$ hence

$$|g^{\bar l i}g^{\bar j \alpha} (\d_{\bar l}g_{\alpha \bar \beta}) g^{\bar \beta k} g_{t\bar j}\d_k \vp|^2=
\begin{cases}
O(|z_k|^2 (|z_k|^2+\ep^2)^{1-4\beta'_k+\beta_k}) &\mbox{if } k=l \\
O( (|z_k|^2+\ep^2)^{1-2\beta'_k+\beta_k}) & \mbox{else}
\end{cases}$$ 
As the coefficients $\beta'_i$ satisfy $\beta'_i \le 1/2$, all the above expressions are uniformly bounded and therefore $ \left| (\nabla^i g^{\bar j k}) g_{t \bar j}\d_k \vp \right |^2 $ is uniformly integrable.\\

%

\end{enumerate}

In conclusion, $||\nabla(g^{\bar j k }g_{t\bar j}\d_k \vp)||_{L^2} \le C$, and therefore we get an $L^2$ estimate 
\begin{equation}
\label{L2:est}
||\nabla \d_t \vp||_{L^2} \le C
\end{equation} 
By the observations at the beginning of the proof, if shows that $\vph$ is has a uniform $L^2$ bound.
\end{proof}

We will also need the following result, which is a rather easy consequence of the above proof:

\begin{prop}
\label{support}
On $X_t$, we have (uniformly in $t$): 
$$\lim_{\ep \to 0} \int_{\bigcup\{|s_i|^2<\ep\}}  |\vph|^2_{\om}\,  \omte^n = 0$$
\end{prop}

\begin{proof}
Recall from the proof above that $ |\vph|^2_{\om}$ is controlled by $|\nabla(\d_t \vp)|^2$, and that we have the relation $\nabla (\d_t \vp)= \nabla(v \cdotp \vp) +\nabla(g^{\bar j k }g_{t\bar j}\d_k \vp)$. It also follows from the proof above that $\nabla(g^{\bar j k }g_{t\bar j}\d_k \vp)$ is dominated by an explicit function independent of $\ep$, say $G$, and such that $\int_{X_t} G^2 \omte^n <+\infty$. By Lebesgue dominated convergence theorem, it follows that  $\int_{\{\sum |s_i|^2<\ep\}}|\nabla(g^{\bar j k }g_{t\bar j}\d_k \vp)|^2 \omte^n$ converges to $0$.

For the other term, once can work separately with the real and imaginary part of $v \cdotp \vp$. Recall equation \eqref{eq:lin3}, that reads $\Delta u = u +R$ if $u=v \cdotp \vp$ and $R$ satisfies an $L^2$ bound. We can consider an cut-off function $\chi_{\ep}$ that vanishes outside  $V_{\ep}:=\bigcup \{ |s_i|^2<\ep\}$ and such that $\int |\nabla \chi_{\ep}|^2 \omte^n \to 0$ when $\ep$ goes to $0$, cf \cite[\S 9]{CGP}. Multiplying the relation satisfied by $u$ above by $\chi_{\ep}u$ and integrating by parts, we get 
\begin{eqnarray*}
\int_{X_t}\chi_{\ep}|\nabla u |^2 \omte^n &\le& \int_{X_t}|u| \cdotp \la \left|\nabla \chi_{\ep}, \nabla u\ra\right| \omte^n + \int_{X_t}\chi_{\ep} |R|\, |u|\, \omte^n \\
& \le &||u||_{\infty} ||\nabla \chi_{\ep}||_{L^2} ||\nabla u ||_{L^2}+C ||u||_{\infty} \cdotp ||R||_{L^2} \mathrm{Vol}\left(V_{\ep} \right)
\end{eqnarray*}
and by the bounds on $ ||u||_{\infty},||\nabla u ||_{L^2}, ||R||_{L^2}$ at our disposal, we get the expected result.
\end{proof}

%

\vspace{2mm}

\begin{coro}
\label{cor:positive}
The current $\rho$ is positive.
\end{coro}

\begin{proof}

Recall from \cite[Eq. (35)]{Paun12} (cf also \cite[Proposition 3]{Schum}) that on $X_t$, 
\begin{equation}
\label{eq:c}
(-\Delta+\mathrm{id})c(\rho_{\ep}) = |\bar \d \vph |^2 + \sum_i(1-\beta_i) \left[\tt+\ddc \log(|s_i|^2+\ep^2)\right](\vph,\vph)+|\vph|^2_{\gamma}
\end{equation}
where $\rho_{\ep}^{n+1}=c(\rho_{\ep}) \rho_{\ep}^n \wedge i dt\wedge d\bar t$, $\Delta = \Delta_{\om_{\vp_{t,\ep}}}$ and $\tt$ is a smooth metric on $D_{i,t}$. One deduces the following inequality:
\begin{eqnarray*}
(-\Delta+\mathrm{id})c(\rho_{\ep}) &\ge& \sum_i(1-\beta_i)\frac{\ep^2}{|s_i|^2+\ep^2}\, \tt(v_{\vp},v_{\vp})\\
&\ge &-\sum_i \frac{C\ep^2}{|s_i|^2+\ep^2} \,|\vph|_{\om}^2
\end{eqnarray*}
for the background Kähler metric $\om$. 
Applying Proposition \ref{minoration}, one obtains for any $z\in X_t$:

\begin{equation}
\label{infc}
 c(\rho_{\ep})(z) \ge -C \int_{X_t}\left(\sum_i\frac{\ep^2 }{|s_i|^2+\ep^2}\right)\cdotp |\vph|_{\om}^2(w)\cdotp  \frac{1}{d_{\om}(z,w)^{2n-2}} \,dV_{\om}(w)
\end{equation}
where $d_{\om}$ is the geodesic distance associated with $\om$ (here and in what follows, replace the exponent $2n-2$ by $1$ if $n=1$). \\

Assume for now that $z\in X_t\ssm D_t$. 

\noindent
Then it follows from \eqref{Linf2} that there exists a a small neigborhood $U_z \Subset X_t \ssm D_t$ and a constant $C=C(z)>0$ such that:\\
\begin{enumerate}
\item[$\bullet$] $|\vp|^2_{\om} \le C$ and $|s_i|^2 \ge C^{-1}$ hold on $U_z$ for any index $i$;\\
\item[$\bullet$] $\int_{U_z} d_{\omte}(z,w)^{2-2n} dV_{\omte}(w) \le C$;\\
\item[$\bullet$] $d_{\omte}(z,w)^{2-2n} \le C$ for any $w\notin U_z$.\\
\end{enumerate}

The second item is a consequence of a simple calculation with the model approximate conic metric. Eventually, one gets:
$$c(\rho_{\ep})(z) \ge -C(z) \left(\ep^2 \int_{U_z}d_{\omte}(z,w)^{2-2n} dV_{\omte}(w) +\int_{X_t \ssm U_z}\left(\sum_i\frac{\ep^2 }{|s_i|^2+\ep^2}\right)\cdotp |\vph|_{\om}^2 \,dV_{\omte}\right)$$
The first term in the right hand side converges to zero because of the second bullet point. As for the second term, one can handle it using Propositions \ref{support} and \ref{prop:L2}.\\
 
Indeed, recall the notation $V_{\ep}=\bigcup \{|s_i| ^2 < \ep\}$, and let us write $N:= \#\{\mbox{components of } D\}$.  One always has $0 \le \sum_i\frac{\ep^2 }{|s_i|^2+\ep^2}\le N$. And on the complement of $V_{\ep}$, one has $ \sum_i\frac{\ep^2 }{|s_i|^2+\ep^2} \le N\ep$. Therefore:

$$\int_{X_t}|\vph|_{\om}^2\left(\sum_i\frac{\ep^2 }{|s_i|^2+\ep^2}\right) \,\omte^n \le N \int_{V_{\ep}} |\vph|_{\om}^2\,\omte^n+N\ep \int_{X_t \ssm V_{\ep}} |\vph|_{\om}^2\,\omte^n$$
In the right hand side, the first term converges to zero by Proposition \ref{support} while the second also converges to zero thanks to Proposition \ref{prop:L2}.\\

In summary, we have proved that if $z\in X_t\ssm D_t$, then 
\begin{equation}
\label{liminf}
\liminf_{\ep \to 0} c(\rho_{\ep})(z)\ge 0
\end{equation}
Even better, if $K\Subset X_t \ssm D_t$, one has 
\begin{equation}
\label{liminf2}
\liminf_{\ep \to 0} \inf_{z\in K} c(\rho_{\ep})(z)\ge 0
\end{equation}

We claim that $\rho$ is the weak limit of $\rho_{\ep}$. Indeed, this is a consequence of the convergence of the potentials on each fiber therefore everywhere on $X$ combined with the uniform $L^{\infty}$ estimate on the potentials cf Proposition \ref{prop:est} allowing one to use Lebesgue dominated convergence theorem.

As a result, $\rho$ is a positive current on $X\ssm D$.  Now $\rho$ has bounded potentials thus $\rho_{|X\ssm D}$ extends trivially to a unique positive current $\bar \rho$ on $X$. So locally, the potentials of $\rho$ and $\bar \rho$ differ by a function $F$ which is locally bounded and pluriharmonic outside $D$, hence $F$ is polyharmonic everywhere and $\rho=\bar \rho$ is a positive current on $X$. 
\end{proof}

\section{Second order estimates in the transverse direction}

The strategy is the same, but the computations get heavier. Here $\vep$ is still the vector field lifting $\frac{\d}{\d t}$ with respect to the metric $\om_{\beta',\ep}$ for some $\beta'=(\beta'_1, \ldots, \beta'_N)$ with $\beta'_i \leq \min \{\beta_i, 1/2\}$. Let us set some notations first. We will work in a trivializing chart with coordinates $(z_1, \ldots, z_n, t)$ where $D=(z_1\cdots z_r=0)$. We will drop the indexes $\ep$ and $t$ to lighten notations once again. The $L^p$ spaces are computed with respect to $\omte^n$, or equivalently $\ombte^n$.
We write $v=\frac{\d}{\d t}+\sum_{i=1}^n v_i \frac{\d}{\d z_i}$ where $v_i = -\sum_jh^{\bar j i}h_{t\bar j}$ if $(h_{\alpha \bar \beta})$ denotes the components of $\omega_{\beta',t,\ep}$ in these coordinates. 
Therefore, we get
$$\bar v \cdotp v = \d_{\bar t t}+v_i \d_{\bar t i}+\bar v_k \d_{\bar k t}+\bar v_kv_i \d_{\bar k i}+(\d_{\bar t}v_i+\bar v_k \d_{\bar k}v_i) \d_i$$
We claim that 
\begin{equation}
\label{interm}
|\! |(\bar v \cdotp v - \d_{\bar t t})(\vp)|\! |_{L^{2}}\le C
\end{equation}
for some uniform $C$. Here the $L^{2}$ norm is taken with respect to $\omte$ on any coordinate chart of $X_t$. To see this, recall that we have 
$$\sum_{k=1}^r(|z_k|^2+\ep^2)^{\beta'_k-1}|v_k|+\sum_{k=r+1}^n |v_k| \le C$$
and by \eqref{prop:est}-\eqref{L2:est}:
$$\sum_{k,i} \left|\!\left|\bar v_kv_i \d_{\bar k i}\vp\right|\!\right|_{L^{\infty}}+ |\!| \sum_{k=1}^r(|z_k|^2+\ep^2)^{1-\beta_k}|\d_{\bar k t} \vp|^2+\sum_{k=r+1}^n  |\d_{\bar k t} \vp|^2|\!|_{L^1}\le C$$
We are left to estimating the $L^2$ norm of the term $(\d_{\bar t}v_i+\bar v_k \d_{\bar k}v_i) \d_i\vp$. 
Thanks to Corollary \ref{est:grad} and \eqref{Linf} we get: 
$$|\d_t\vp|+ \sum_{k=1}^r(|z_k|^2+\ep^2)^{1-\beta_k} |\d_k \vp|^2 +\sum_{k=r+1}^n |\d_k \vp| \le C$$ 
As $\d_k v_i=- g^{i \bar \alpha} g_{\alpha \bar \beta, k}g^{\beta \bar j}g_{t\bar j}-g^{i \bar j}g_{t\bar j,k}$ and all the coefficients $\beta'_{\bullet}$ are $\le 1/2$, it is easy to show that
\begin{equation}
\label{vder}
\d_k v_i = \begin{cases}
O(z_i) & \mbox{if } i=k \in \{1, \ldots, r\} \\
O(z_i z_k (|z_k|^2+\ep^2)^{\beta'_k-1}) & \mbox{if } i,k \in \{1, \ldots, r\}, i \neq k \\
O(z_k (|z_k|^2+\ep^2)^{\beta'_k-1}) & \mbox{if } k \in \{1, \ldots, r\}, i \notin \{1, \ldots, r\} \\
O(z_i) & \mbox{if } i \in \{1, \ldots, r\}, k\notin \{1, \ldots, r\}  \\
O(1)  &\mbox{else}
\end{cases}
\end{equation}
Therefore, $\d_{\bar t}v_i+\bar v_k \d_{\bar k}v_i=O(1)$ if $i\notin \{1, \ldots, r\}$ and $\d_{\bar t}v_i+\bar v_k \d_{\bar k}v_i=O(z_i)$. As a result, 
$$|(\d_{\bar t}v_i+\bar v_k \d_{\bar k}v_i) \d_i\vp|^2 = \begin{cases}
O((|z_i|^2+\ep^2)^{\beta'_i}) & \mbox{if } i \in \{1, \ldots, r\}\\ 
O(1) & \mbox{else } 
\end{cases}$$
which prove \eqref{interm}.
\begin{prop}
\label{2nd}
There exists $C>0$ independent of $\ep,t$ such that 
$$\left|\int_{X_t} (\bvep \cdotp \vep \cdotp \vp_{t,\ep}) \, \omte^n \right| \le C$$
\end{prop}
\begin{proof}
We start from the equation $\omte^n= \frac{e^{\vp_{t,\ep}}}{\prod (|s_i|^2+\ep^2)^{1-\beta_i}}\om^n$. Recall that the volume of $\omte$ is constant, hence $\frac{\d^2}{\d t\d \bar t} \int_{X_t}\frac{e^{\vp_{t,\ep}}}{\prod (|s_i|^2+\ep^2)^{1-\beta_i}}\om^n=0$.
By Proposition \ref{fint}, we get:
\begin{equation}
\label{div2}
\int_{X_t} (\vep \cdotp \bvep \cdotp \vp_{t,\ep}) \, \omte^n =   \sum_i (1-\beta_i) \int_{X_t} (\vep \cdotp \bvep \cdotp \log(|s_i|^2+\ep^2)) \, \omte^n
\end{equation}
We are reduced to studying the behavior of each function $\vep \cdotp \bvep \cdotp \log(|s_j|^2+\ep^2))$, and it suffices to show that on each coordinate chart, this function is uniformly integrable. Once again we are going to drop the $\ep$.  
Remember that $\bar v \cdotp v = \d_{\bar t t}+v_i \d_{\bar t i}+\bar v_k \d_{\bar k t}+\bar v_kv_i \d_{\bar k i}+(\d_{\bar t}v_i+\bar v_k \d_{\bar k}v_i) \d_i$. We decompose $\bar v \cdotp v \cdotp \log(|s_j|^2+\ep^2)$ into three terms that we will evaluate at $p_0$, the center of the coordinate chart. 

\begin{enumerate}
\item[$a.$] The term $(\d_{\bar t}v_i+\bar v_k \d_{\bar k}v_i) \d_i \log(|s_j|^2+\ep^2)$.

\noindent
At $p_0$, we have $$\d_i \log(|s_j|^2+\ep^2)= \begin{cases}
\frac{\bar z_i}{|z_i|^2+\ep^2} & \mbox{if } i=j \\
0 & \mbox{else}
\end{cases}$$
and $(\d_{\bar t}v_i+\bar v_k \d_{\bar k}v_i)=O(z_i)$ if $i\in \{1, \ldots, r\}$ so that the term we are estimating is a $O(1)$.\\

\item[$b.$] The term $(\d_{\bar t t}+v_i \d_{\bar t i}+\bar v_k \d_{\bar k t}) \log(|s_j|^2+\ep^2))$.

\noindent
We start with the formula, valid at $p_0$:
 $$\ddc \log (|s_j|^2+\ep^2)= \frac{\ep^2 idz_j \wedge d\bar z_{j}}{(|z_j|^2+\ep^2)^2}- \frac{|z_j|^2}{|z_j|^2+\ep^2} \, \theta_{D_i}$$
from which is it obvious that our term is bounded (in our coordinate system, $t=z_{n+1}$ and $j<n+1$.\\

\item[$c.$] The term $\bar v_kv_i \d_{\bar k i} \log(|s_j|^2+\ep^2))$.

\noindent
By the previous formula, all we have to care about is when $i=j=k$ in which case our term is a $O(\ep^2(|z_i|^2+\ep^2)^{-2\beta'_i})$ which is bounded as $\beta'_i \le 1/2$. \\
\end{enumerate}

\noindent
To summarize, we have proved that $\sum_i (1-\beta_i) (\vep \cdotp \bvep \cdotp \log(|s_i|^2+\ep^2))$ is uniformly bounded. Combining this with \eqref{div2}, we obtain expected the bound.
\end{proof}

\begin{coro}
\label{cor:L12}
There exists a uniform constant $C>0$ such that: 
$$\int_{X_t} |c(\re)| \,  \omvpe^n \le C.$$ 
\end{coro}

\begin{proof}
Given the definition of $c(\rho_{\ep})$, if follows from on \eqref{L2:est} that on any trivializing chart $U_t$, one has $||c(\rho_{\ep}) - \d_{t\bar t}\, \vp_{t,\ep}||_{L^2(U_t)} \le C$. Combining this with \ref{interm}, one gets 
\begin{eqnarray}
\label{error}
||c(\rho_{\ep}) - \vep \cdotp \vep \cdotp \vp_{t,\ep}||_{L^2(X_t)} \le C
\end{eqnarray}
For $C$ big enough, the positive function $F_{\ep}(z):=C\int_{X_t}|\vph|^2d_{\omte}(z,w)^{2-2n}dV(w)$ satisfies 
\begin{enumerate}
\item[$\bullet$] $c(\rho_{\ep}) \ge -F_{\ep}$
\item[$\bullet$] $\int_{X_t} F_{\ep}dV \le C'$
\end{enumerate}
Here $dV$ is the volume form associated with $\omte$. The first item is a reformulation of \ref{infc} while the second follows from the uniform integrability of $d_{\omte}(\cdotp,w)^{2-2n}$ combined with Fubini theorem and Proposition \ref{prop:L2}.

One has: 
$$\int_{X_t}  \bar v_{\ep} \cdotp \vep \cdotp \vp_{t,\ep}\, dV= \int_{X_t}(\bar v_{\ep} \cdotp \vep \cdotp \vp_{t,\ep}-c(\re))dV+\int_{X_t}(c(\rho_{\ep}) +F_{\ep})dV-\int_{X_t}F_{\ep}\, dV$$
Proposition \ref{2nd} shows that $\left|\int_{X_t}  \bar v_{\ep} \cdotp v_{\ep} \cdotp \vp_{t,\ep}\right| \le C$. From \eqref{error} and the second bullet item, one infers that the integral $\int_{X_t}(c(\rho_{\ep})+F_{\ep})dV$ is uniformly bounded. Because $c(\re)+F_{\ep}\ge0$, this yields:
$$||c(\re)||_{L^1}\le ||c(\re)+F_{\ep}||_{L^1}+||F_{\ep}||_{L^1}\le C$$ 
This proves the corollary.
\end{proof}

\vspace{3mm}
\noindent
From there, one can easily deduce the boundedness of the current $\rho$ outside of the divisor $D$: 
\vspace{3mm}

\begin{coro}
\label{cor:reg}
The current $\rho$ is bounded outside the divisor $D$; in particular, its coefficients are locally bounded functions and its local potentials belong to $\mathscr C^{1,\alpha}_{\rm loc}(X_0\ssm D)$ for any $\alpha<1$. Furthermore the coefficients of $\rho$ are smooth along the fiber directions outside $D$. 
\end{coro}

\begin{proof}
Let $U\subset X_t$ be a coordinate chart which does not intersect $D$. We aim to show that on this set, $\d_{t \bar t} \vp_{t,\ep}$ admits uniform $\mathscr C^{k}$ estimates for all $k\ge 0$ independent of $\ep$ and $t$ (small enough). The case $k=0$ is crucial, as one shall see.\\

Let us first reduce the $\mathscr C^k$ estimate to the $\mathscr C^0$ one. Note that because of the estimates \eqref{Linf2} on $\d_t \vp_{t,\ep}$ and the expression of $c(\rho_{\ep})$ in coordinates recalled in \S \ref{diff}, it is enough to prove $\mathscr C^k$ estimates for $c(\rho_{\ep})$, which is solution of 
$$(\Delta-\mathrm{id})c(\rho_{\ep}) = -|\bar \d \vph |^2 - \sum_i(1-\beta_i) \left[\tt+\ddc \log(|s_i|^2+\ep^2)\right](\vph,\vph)-|\vph|^2_{\gamma}$$
Thanks to \eqref{Linf2}, $\vph$ admits uniform $\mathscr C ^k$ estimates on $U$; and of course, so does $\omvpe$. So by Schauder estimates, a $\mathscr C^0$ estimate on $c(\rho_{\ep})$ will yield  $\mathscr C ^k$ estimates on $U$ for $c(\re)$ and thus $\d_{t\bar t}\vp$.\\

For the $\mathscr C^0$ estimate, the important informations available to us are the $L^1$ bound for $c(\rho_{\ep})$ established in Corollary \ref{cor:L12} and the local lower bound for $c(\re)$ on the compact sets of $X \ssm D$, cf \eqref{liminf2}. 

Now, let us choose a base point $x\in U$, fix $R>0$ such that $B(4R)\subset U$, and choose a positive number $\delta<\frac 2 {n-2}$. From Harnack inequality \cite[Theorem 8.17]{Gilb}, one gets
$$\sup_{B(R)} c(\rho_{\ep}) \le C(||c(\rho_{\ep})||_{L^{1+\delta}(B(2R))}+1)$$
but \cite[Theorem 8.18]{Gilb} yields:
$$||c(\rho_{\ep})||_{L^{1+\delta}(B(2R))} \le C(\inf_{B(R)} c(\rho_{\ep})+1)$$
Actually Harnack inequality applies to non-negative functions; however, one can add to $c(\re)$ a fixed constant to make it positive (cf \eqref{liminf2}) and this operation leaves the estimates essentially unchanged.
Therefore, if $\limsup_{\ep \to 0} \sup_{B(R)} c(\rho_{\ep}) = +\infty$, then $\limsup_{\ep \to 0} \inf_{B(R)} c(\rho_{\ep}) = +\infty$, which contradicts the $L^1$ control we have on $c(\rho_{\ep})$. Therefore $c(\rho_{\ep})$ is uniformly bounded above, but we already knew the bound from below, cf \eqref{liminf2}.  Hence the $\mathscr C^0$ estimate. \\

As the estimate is uniform in $t$ small enough, one gets that for any point $x\in X_0\ssm D$, there exist a neighborhood $\Omega\subset X_0\ssm D $ of $x$ and a constant $C>0$ such that on $\Omega$, one has
$$\re \le C \om$$
for some background Kähler form $\om$ on $X$. As $\rho$ is the weak limit of $\re$ when $\ep \to 0$ (cf the few lines below \eqref{liminf2}), one gets the boundedness of $\rho$ outside $D$. Therefore the coefficients of $\rho$ outside $D$ belong to the dual of $L^{1}$, hence they are locally bounded \textit{functions}. 

Finally, the claim on $\mathscr C^{1,\alpha}$ regularity follows easily from the boundedness of $\d_{t\bar t}\vp$. Indeed, on a coordinate chart of $X_0\ssm D$, the Laplacian of $\vp_{t,\ep}$ is bounded by the $\mathscr C^0$ estimate above combined with Proposition \ref{prop:est}. By standard results (e.g. \cite[Theorem 3.9]{Gilb}) this yields $\mathscr C^{1,\alpha}$ bounds on $\vp_{t,\ep}$ for any $\alpha<1$ and from there, Arzela-Ascoli theorem shows the expected regularity of $\vp$.
\end{proof}

\section{Extension}
\label{extension}
At this point, one knows that the fiberwise twisted conic Kähler-Einstein metrics induce a closed positive $(1,1)$-current  $\rho \in c_1(K_{X/Y}+\sum(1-\beta_k)D_k)_{|X_0}+\{\gamma\}_{|X_0}$ on $X_0$; we would like to extend that current to a positive current on $X$. To achieve this goal, it is sufficient to prove that the local potential of $\rho$
are is bounded from above near the singular fibers. To do so, one follows the strategy in \cite[\S 3.3]{Paun12} which can be carried out to this more general setting without significant change.

We pick a point $x_0$ in $X_0=\pi^{-1}(x_0)$, and choose a Stein neighborhood $\Omega$ of $x_0$ in $X$; we write $\Omega_y = \Omega \cap X_y$, choose
a potential $\tau_y$ of $\rho_y$ so that (up to adding a pluriharmonic function to $\tau_y$) the equation satisfied by $\tau_y$ on $\Omega_y$ is
$$(\ddc \tau_y)^n = \frac{1}{\prod |f_k|^{2(1-\beta_k)}}e^{\tau_{y}-F}\left| \frac {dz}{dt} \right|^2$$
if $(f_k=0)$ is an equation of $D_k$, $F$ is a local potential for $\gamma$ (on $\Omega$), and the coordinates $(z_1, \ldots, z_n, t_1, \ldots, t_m)$ are chosen so that $p(\underline z, \underline t)=\underline t$.
We set
$$H_{m,y}:=\left\{f\in \mathcal O(\Omega_y); \int_{\Omega_y}|f|^2e^{-m\tau_y}(\ddc \tau_y)^n \le 1\right\}$$
Then $$\tau(y)(x_0)= \lim_{m\to \infty}\sup_{f\in H_{m,y}} \frac 1 m \log |f(x_0)|$$
But for $f\in H_{m,y}$, Hölder's inequality yields 
\begin{equation}
\label{maj}
\int_{\Omega_y}|f|^{2/m}e^{-\tau_y}(\ddc \tau_y)^n \le \left(\{\rho_{|X_y}\}^{n}\right)^{\frac{m}{m-1}}
\end{equation}
and the right hand side is bounded above independently of $y$ and $m$. Furthermore, the $L^{2/m}$ version of Ohsawa-Takegoshi extension theorem \cite{BP2} yields a holomorphic function $F$ on $\Omega$ that extends $f$ and such that
$$|F(x_0)|^{2/m} \le C_{\Omega} \int_{\Omega} |F|^{2/m} |dz|^2 \le C \int_{\Omega_y}|f|^{2/m}\left| \frac {dz}{dt} \right|^2\le C' \int_{\Omega_y}|f|^{2/m} e^{-\tau_y+F}(\ddc \vp_y)^n$$
as $\prod |f_k|^{2(1-\beta_k)}$ is uniformly bounded above. Moreover, the integral on the right hand side is bounded above uniformly in $y$ and $m$ by  \eqref{maj} and the smoothness of $F$. 
Therefore $\tau_y \le C$ on $\Omega_y$ for a constant $C$ independent of $y \in p(\Omega\cap X_0)$.

\section{Some remarks}

\subsection{The cuspidal case}
First, one can generalize the final conclusion of the main Theorem in the case where the boundary divisor $D=\sum (1-\beta_k) D_k$ has coefficients $\beta_j \in [0,1)$. Indeed, if $c_1(K_{X_y}+\sum_{k=1}^r (1-\beta_k){D_k}_{|{X_y}})+ \{\gamma\}_{|X_y}$ is Kähler for a generic $y\in Y$, then so is $ c_1(K_{X_y}+(1-\ep)\sum_{k=1}^r (1-\beta_k){D_k}_{|{X_y}})+ \{\gamma\}_{|X_y}$ for $\ep>0$ small enough.  Therefore, $ c_1(K_{X/Y}+ \sum_{k=1}^r(1-\beta_r) D_k)+\{\gamma\}$ is pseudoeffective as limit of pseudoeffective classes. The question of the psh variation of the associated conic/cuspidal Kähler-Einstein metric seems to be more involved though. In the case where the divisor is reduced, it seems very plausible that one could apply the implicit function theorem in the Cheng-Yau Hölder spaces to get the "smoothness" of the fiberwise Kähler-Einstein metric, and from there the positivity of its variation. 

%
%

\subsection{More regularity?}
One may ask whether one can obtain a better regularity for the global potential $\vp$ of $\rho$ on the locus $X_0$. Without loss of generality, one can assume that $p:X \to \mathbb D$ is smooth.There are two types of improvement one could be looking for: \\

$\bullet$ Away from $D$: we only obtained a control of the first two mixed derivatives $\d_t \vp$ and $\d_{t \bar t} \vp$, which we showed to be smooth \textit{along the fiber directions}. Understanding the regularity of higher transverse derivatives of $\vp$ seems to require some significant additional work. \\

$\bullet$ Near $D$:  we have proved that $\vp, \d_t \vp$ are bounded. Given the simple elliptic differential equation satisfied locally by $\d_t \vp$ on each fiber: $(\Delta_{\om_t}- \mathrm{Id})(\d_t \vp)=(\tr_{\alpha_t}-\tr_{\om_t})(\d_t \om)+\d_t \log |s|^2$ 
then one can use the conic Harnack inequalities \cite[Theorem 7.11]{GP} combined with Remark 7.13 in the same paper to get that $\d_t \vp \in  \mathscr C^{\alpha, \beta}(X_t)$; from there, and provided the conic Schauder estimates obtained in \cite{Don} extend to the snc case, the function $\d_t \vp$ would belong to the space $ \mathscr C^{2,\alpha, \beta}(X_t)$. 

The main problem though is about $\d_{t \bar t} \vp$. If one can show that this quantity is uniformly bounded outside $D$ (or equivalently that $c(\rho)$ is bounded on $X_t$), then it would mean that $\rho$ is dominated by a global conic metric on $X$. Let us make this a bit more explicit. Near a point $x_0\in X_0$, one can choose a neigborhood $\Omega_0\subset X_0$ with local holomorphic coordinates $(z_1, \ldots, z_n, t_1, \ldots, t_m)$ on $X$ such that $p(\underline z,  t) =  t$, and such that up to relabelling the cone angles, $\sum (1-\beta_k) D_k$ is given by $(1-\beta_1) [z_1=0]+ \cdots + (1-\beta_s) [z_s=0]$. Then provided $c(\rho)$ is bounded, there would exist a constant $C$ (depending on $\Omega_0$) such that 
$$0 \le \rho \le C\left(\sum_{k=1}^s \frac {idz_k\wedge d\bar z_{k}}{|z_k|^{2(1-\beta_k)}}+\sum_{k=s+1}^n i dz_k \wedge d \bar z_k+  i dt \wedge d\bar t\right)$$  
holds uniformly on $\Omega_0$. 

From there, it would be easy to conclude that $\d_{t\bar t} \vp \in \mathscr C^{2,\alpha, \beta}(X_t)$, so that the coefficients of $\rho$ would have the (refined) regularity of a conic metric on $X_t$. Essentially, one can differentiate the equation satisfied by $\d_{t\bar t}\vp$ again, and because $\d_t \vp$ is $\mathscr C^{2,\alpha, \beta}$ and $\d_{t\bar t}\vp$ is bounded, Harnack's inequality \cite{GP} would show that $\d_{t\bar t} \vp$ is $\mathscr{C}^{\alpha, \beta}$ and from there, Schauder estimates \cite{Don} would yield the claim.  

So essentially all the remaining regularity issues on the coefficients of $\rho$ in the fiber directions can be brought down to the boundedness of $c(\rho)$ across $D$ near the regular fibers, which by standard arguments reduces to showing a $L^{p}$ estimate on $c(\re)$ for some $p>1$ (whereas we are only able to give such an estimate for $p=1$). 

\bibliographystyle{smfalpha}
\bibliography{biblio}

\end{document}